\documentclass[a4paper, 11pt]{amsart}
\usepackage{verbatim}
\usepackage{graphicx}
\usepackage{float}
\usepackage{amsmath, amssymb, amsthm}

\newtheorem{theorem}{Theorem}[section]
\newtheorem{acknowledgement}[theorem]{Acknowledgement}

\newtheorem{corollary}[theorem]{Corollary}

\newtheorem{definition}[theorem]{Definition}

\newtheorem{lemma}[theorem]{Lemma}

\newtheorem{proposition}[theorem]{Proposition}
\newtheorem{remark}[theorem]{Remark}

\newcommand{\bh}{\mathcal{B}(\mathcal{H})}
\newcommand{\os}{\mathcal{O}(\mathcal{S})}

\topmargin by -\headsep
\title{The Schur-Horn theorem in von Neumann algebras}
\begin{document}
\date{}
\author{Mohan Ravichandran}
\address{Mimar Sinan university, Bomonti, Istanbul, Turkey - 34380}
\email{mohan.ravichandran@gmail.com}
\maketitle

\begin{center}

Dedicated to the memory of William Arveson(1934-2011)

\end{center}

\begin{abstract}
A few years ago, Richard Kadison thoroughly analysed the diagonals of projection operators on Hilbert spaces and asked the following question: Let $\mathcal{A}$ be a masa in a type $II_1$ factor $\mathcal{M}$ and let $A \in \mathcal{A}$ be a positive contraction. Letting $E$ be the canonical normal conditional expectation from $\mathcal{M}$ to $\mathcal{A}$, can one find a projection $P \in \mathcal{M}$ so that 
\[E(P) = A?\]
In a later paper, Kadison and Arveson, as an extension, conjectured a Schur-Horn theorem in type $II_1$ factors. In this paper, I give a proof of this conjecture of Arveson and Kadison. I also prove versions of the Schur-Horn theorem for type $II_{\infty}$ and type $III$ factors as well as finite von Neumann algebras. 

\end{abstract}
\section{Introduction}
The classical Schur Horn theorem\cite{SchPap}, \cite{HorPap}, relates the diagonal and eigenvalue lists of a hermitian matrix: Let $A$ be a positive semidefinite element of $M_{n}(\mathbb{C})$ and let $d = (d_1, d_2, \cdots, d_n)$ and $\lambda = (\lambda_1, \lambda_2, \cdots, \lambda_n)$ be the lists of diagonal entries and eigenvalues respectively, both sorted in non-increasing order. Then, the Schur-Horn theorem says that we must have 
\[d_1 + \cdots + d_k \leq \lambda_1 + \cdots + \lambda_k \quad  1 \leq k \leq n \quad \text{and} \quad  d_1 + \cdots + d_n = \lambda_1 + \cdots + \lambda_n.\] 

The above condition on the lists is denoted by saying that the diagonal list is majorized by the eigenvalue list, written $d \prec \lambda$. The Schur-Horn theorem states that further, given two positive lists $d, \lambda$ with $d \prec \lambda$, then there is a positive semi-definite matrix $A$ with eigenvalues $\lambda$ and diagonal $d$.

Majorization can also be defined for matrices. Given two positive operators $A, S$ in $M_{n}(\mathbb{C})$, we say that $A \prec S$ if the eigenvalue sequence of $A$ is majorized by the eigenvalue sequence of $S$. The Schur-Horn theorem can then be stated as saying that if $A$ is a diagonal positive matrix and $S$ a positive matrix so that $A\prec S$, then there is a unitary operator $U$ so that the diagonal of $USU^{*}$ is $A$. 

Majorization for matrices has the following alternate description due to Hardy, Littlewood and Polya\cite{HLPIne},
\begin{definition}[Majorization]
Given two self-adjoint operators $A, S$ in $M_n(\mathbb{C})$,  $A$ is majorized by $S$ iff 
\[\operatorname{Tr}(f(A)) \leq \operatorname{Tr}(f(S))\] for every continuous convex real valued function $f$ defined on a closed interval $[c,d]$ containing the spectra of both $A$ and $S$. 
\end{definition}
Majorization in type $II_1$ factors\cite{HiaiMa} is described analogously, with the trace on $M_n(\mathbb{C})$ in the definition replaced by the canonical trace $\tau$. 

Let $\mathcal{A}$ be a maximal abelian sefladjoint subalgebra(masa, in short) in a type $II_1$ factor $\mathcal{M}$; There is a unique trace preserving normal(weak* to weak* continuous) conditional expectation $E:\mathcal{M} \rightarrow \mathcal{A}$ that is in many ways analogous to the restriction mapping onto the diagonal for elements of $M_{n}(\mathbb{C})$. Arveson and Kadison\cite{ArvKad} showed that if $S$ is a positive operator in $\mathcal{M}$, then $E(S) \prec S$. This fact can also be deduced from Hiai's work on stochastic maps on von Neumann algebras\cite{HiaiMa}.

There are two natural generalizations of the Schur-Horn theorem to type $II_1$ factors. The first  originates in the standard interpretation of the Schur-Horn theorem as characterizing the set of all possible diagonals of a positive matrix. Let $\mathcal{U}(\mathcal{M})$ be the set of unitary operators in $\mathcal{M}$ and given an operator $S$, let $\mathcal{O}(S)$ be the norm closure of the unitary orbit of $S$, i.e 
\[\mathcal{O}(S) = \overline{\{U S U^{*}  \mid U \in \mathcal{U}(\mathcal{M})\}}^{||}.\] 
Two positive operators $A$ and $S$ in a type $II_1$ factor $\mathcal{M}$ are said to be \emph{equimeasurable}, denoted $A \approx S$, if $\tau(A^n) = \tau(S^n)$ for $n = 0, 1, 2, \cdots$. It is routine to see that the following are equivalent. 
\begin{enumerate}\label{EM}
 \item $A \approx S$. 
 \item  $A \in \os$. 
\end{enumerate}
The following result characterizing possible ''diagonals'' of positive operators, is the first main theorem in this paper.

\newtheorem*{theorem:associativity}{Theorem \ref{conj1}}
\begin{theorem:associativity}[The Schur-Horn theorem in type $II_1$ factors I]
Let $\mathcal{M}$ be a type $II_1$ factor and let $A, S \in \mathcal{M}$ be positive operators with $A \prec S$. Then, there is some masa $\mathcal{A}$ in $\mathcal{M}$ such that \[E_{\mathcal{A}}(S) \approx A.\]
\end{theorem:associativity}

The second generalization was conjectured by Arveson and Kadison in \cite{ArvKad}.  The second main theorem in this paper is the proof of their conjecture,

\newtheorem*{theorem:associativity2}{Theorem \ref{conj2}}
\begin{theorem:associativity2}[The Schur-Horn theorem in type $II_1$ factors II]
Let $\mathcal{A}$ be a masa in a type $II_1$ factor $\mathcal{M}$. If $A \in \mathcal{A}$ and $S \in \mathcal{M}$ are positive operators with $A \prec S$. Then, there is an element $T \in \mathcal{O}(S)$ such that \[E(T) = A\]
\end{theorem:associativity2}

One cannot escape having to take the norm closure of the unitary orbit of $S$, see the paper \emph{loc.cit. }for a discussion on the necessity. In infinite dimensions, unitary equivalence cannot be determined from spectral data alone. On another note, it is trivial to see that the above theorems about diagonals for positive operators immediately yield identical theorems for hermitian operators, by adding a suitable constant to make them positive. 

A special case of the above theorem, namely, that given any positive contraction $A$ in $\mathcal{A}$, there is a projection $P$ in $\mathcal{M}$ so that $E(P) = A$, had been conjectured earlier by Kadison in \cite{KPNAS1}, see also \cite{KPNAS2}, who referred to it as the ''carpenter'' problem in type $II_1$ factors.

\begin{remark}
Neither of the two theorems directly implies the other. It is however, easy to see that theorem(\ref{conj2}) implies theorem(\ref{conj1}) when $S$ has finite spectrum.
\end{remark}

There has been methodical progress towards the resolution of Arveson and Kadison's conjecture (\ref{conj2}): Argerami and Massey\cite{ArMaSH} showed that 
\[\overline{E(\mathcal{O}(S))}^{\operatorname{SOT}} = \{A \in \mathcal{A} \mid A \prec S\}\] This was improved by Bhat and Ravichandran\cite{BhaRav}, who showed that it is enough to take the norm closure. They also showed that the conjecture holds when both the operators $A$ and $S$ have finite spectrum. Dykema, Hadwin, Fang and Smith\cite{DFHS} gave a natural way to approach the problem and reduced the conjecture to a question involving kernels of conditional expectations. Using this approach, they were able to show that the conjecture holds, among other cases, for the radial and generator masas in the free group factors. However, it is unclear if their strategy can be be used to settle the conjecture in full.

It must be pointed out that approximate Schur-Horn theorems are easier to obtain than exact ones. Further, it is possible that one might lose much fine structure: For instance, Kadison characterised the diagonals of projections in $\mathcal{B}(\mathcal{H})$ and discovered an index obstruction to a sequence arising as the diagonal of a projection. This subtlety is lost when ones passes to the norm closure of the set of diagonals, see \cite{ArDiPN} for a discussion.

There has been a great deal of progress towards characterising the diagonals of hermitian operators in  $\mathcal{B}(\mathcal{H})$; Unlike the finite dimensional case, the situation in infinite dimensions is highly subtle. Bownik and Jasper have given a complete description of the possible diagonals of hermitian operators with finite spectrum in $\mathcal{B}(\mathcal{H})$ in \cite{BowJas1} and \cite{BowJas2}. In another direction, the diagonals of compact operators have been characterised by Loreaux and Weiss in \cite{LorWei}, extending the earlier work of Kaftal and Weiss in \cite{KafWei}.

I end with a recent development; Thompson's theorem in matrix algebras characterises operators which can realised to have a  prescribed diagonal in terms of their singular values.  Kennedy and Skoufranis in \cite{KenSko} have recently extended Thompson's theorem to the setting of type $II_1$ factors. 

This paper has six sections apart from the introduction. In section (2), I collect some standard facts in noncommutative measure theory. Section (3) exploits the useful observation that once we can solve the problem ''locally'', theorem (\ref{conj1}) follows using transfinite induction. Section (4) contains the proof of theorem (\ref{conj1}). Section (5) then builds upon this result to prove theorem (\ref{conj2}). We then use the Schur-Horn theorem for type $II_1$ factors to deduce theorems for type $II_{\infty}$ factors in Section (6). In the last section, namely Section (7) for sake of completeness, I explain the situation both in the case of general finite von Neumann algebras and type $III$ factors.

\begin{acknowledgement}
 I would like to thank Junsheng Fang for telling me about this problem and for several useful discussions.  
\end{acknowledgement}

\section{Notation and basic relationships}
There is a concrete description of majorization in type $II_1$ factors that is more convenient to work with, that we now describe. We will use the following nonstandard definition repeatedly: Given two subsets $X$ and $Y$ of $\mathbb{R}$, say that $X \geq Y$  if $X$ is to the right of $Y$, i.e. $\operatorname{inf}_{x \in X}  \geq  \operatorname{sup}_{x \in Y}$. We analogously define the relation $X>Y$. Also, given a self-adjoint operator $S$, we will use $\alpha(S)$ to denote $\operatorname{inf}\{x \in \sigma(S)\}$.

 Let $A$ be a positive operator in type $II_1$ factor. By the spectral theorem, there is a Borel measure with compact support, $\mu$ on $\mathbb{R}$ so that 
\[\tau(A^n) = \int_\mathbb{R} x^n d\mu \quad n = 0, 1, \cdots\]
Define the real valued function $f_A$ on $[0,1)$ by 
\[f_A(x) = \inf\{t \mid \tau(E_{A}((t,\infty))) \leq x\}.\]
 This function $f_A$ is non-increasing and right continuous. We have the identity  
\[\tau(A^{n}) =\int_0^1 f_A(x)^n dm \quad n = 0,1, \cdots \]
where $m$ denotes Lebesgue measure. The values of this function were denoted the generalised $s$ numbers of $A$ by Fack and Kosaki\cite{FacKos}. We however, choose to call the function $f_A$ the spectral scale of $A$. 

It is a standard fact\cite{Fack} that one can find a projection valued measure which we denote by $\mu_A$ on $[0,1]$ so that $\tau(\mu_A(X)) = m(X)$ for any Borel measurable set $X \subset [0,1]$ and so that 
\begin{eqnarray}\label{SM}
A = \int f_A(t) d\mu_A(t) 
\end{eqnarray}

This projection valued measure is not unique when there are atoms in the spectrum of $A$. However, given a positive operator $A$, we will fix a measure once and for all and use $\mu_A$ to denote this. Throughout this paper, it will be evident that the results will not depend on the particular choice of measure in this degenerate case.

Given two positive operators $A$ and $S$ inside a type $II_1$ factor $M$ with spectral  scales $f_A$ and $f_S$ repectively, it can be shown that $S$ majorizes $A$, written $A  \prec S$ if
\begin{eqnarray}\label{Maj21}
\int_0^r f_A(x)dm \leq \int_0^r f_S(x)dm, \,\, 0 \leq r \leq 1,   \quad \int_0^1 f_A(x)dm = \int_0^1 f_S(x)dm 
\end{eqnarray}
When we do not have the last trace equality, we say that $S$ submajorizes $A$ and denote this by $A \prec_{w} S$.

There are two concise ways of representing these inequalities in type $II_1$ factors. The first uses the Ky Fan norm functions are defined by
\begin{eqnarray}\label{defF}
 F_A(x) := \int_0^x f_A(t) dm(t) \quad \text{ for } \quad 0 \leq x \leq 1
\end{eqnarray}
 The function $F_A$ is continuous and we have that, $A \prec_{w} S$ iff $F_A(x) \leq F_S(x)$ for $x \in [0,1]$. If we also have that $F_A(1) = F_S(1)$, then, $A \prec S$.  Alternately, $A \prec S$ if 
\begin{eqnarray}
\tau(A \mu_A([0,t])) \leq \tau(S \mu_S([0,t])) \text{ for } 0 < t < 1 \quad \text{ and } \quad \tau(A) = \tau(S) 
\end{eqnarray}

Given two positive operators $A$ and $S$ in a type $II_1$ factor $\mathcal{M}$, define the quantity $L_{\mathcal{M}}(A,S)$, also denoted simply by $L(A,S)$ when the ambient algebra is clear, by 
 \begin{eqnarray}
L(A,S): = \min_{0 \leq t \leq 1} (F_S(t) - F_A(t)).
\end{eqnarray} 
We have that $A \prec_{w} S$ exactly when  $L(A,S) = 0$ and the function $L(A,S)$ measures how far $S$ is from submajorizing $A$. We record some facts about the quantity $L(A,S)$. 

\begin{lemma}\label{Llemma}
Let $A, S$ be positive operators in a type $II_1$ factor $\mathcal{M}$. Then, 
\begin{enumerate}
 \item  $L(A,S) \geq - \tau(A)$.
 \item  If $A$ and $S$ commute with a set of orthogonal projections $\{ P_1, \cdots, P_k\}$ which sum up to $I$, then,
\[L(A,S) \geq \sum_{1}^{k} \tau(P_m) L_{P_m \mathcal{M} P_m}(AP_m,SP_m)\]
 \item Suppose additionally that
\[\sigma_{P_1 \mathcal{M}P_1}(AP_1) \geq  \cdots \geq \sigma_{P_k \mathcal{M}P_k}(AP_k),\quad \sigma_{P_1 \mathcal{M}P_1}(SP_1)\geq \cdots \geq \sigma_{P_k \mathcal{M}P_k}(SP_k)\]
Then, letting  $Q_0 = 0$ and $Q_m = P_1 + \cdots + P_m$ for $m = 1, \cdots, k$,
\[L(A,S) = \operatorname{min}_{1 \leq m \leq k} \, \tau((S-A)Q_{m-1})+\tau(P_m)L_{P_m \mathcal{M} P_m}(AP_m,SP_m)\]
\end{enumerate}
\end{lemma}
\begin{proof} It is easy to see that $F_A(0) = 0$ and that $F_A(1) = \tau(A)$. For the first assertion, we have
\begin{eqnarray*}
 L(A,S): = \min_{0 \leq t \leq 1} (F_S(t) - F_A(t)) \geq \min_{0 \leq t \leq 1} F_S(t) - \max_{0 \leq t \leq 1} F_A(t) \geq -\tau(A)
\end{eqnarray*}
For the second, it is easy to see that once we have proved the assertion for $k=2$, the general case follows by induction. Assume then, that $k=2$. Let $0 < t < 1$ be arbitrary; We may write 
\[\mu_S([0,t]) = \mu^{P_1\mathcal{M}P_1}_{SP_1}([0,a]) \oplus \mu^{P_2\mathcal{M}P_2}_{SP_2}([0,b])\]and
\[  \mu_A([0,t]) = \mu^{P_1\mathcal{M}P_1}_{AP_1}([0,c]) \oplus \mu^{P_2\mathcal{M}P_2}_{AP_2}([0,d]),\] for some $a,b,c,d$. Here, the notation $\mu_{AP}^{P\mathcal{M}P}$ means that we calculate the relevant spectral projection for the operator $AP$ considered as an operator inside the $II_1$ factor $P\mathcal{M}P$. The expressions $f_{AP}^{P\mathcal{M}P}$ and  $F_{AP}^{P\mathcal{M}P}$ when referring to the spectral scale and the Ky Fan norm function will be used similarly. 

 Suppose that $a > c$ - The complementary case is handled similarly. We have that
 \[t = \tau(\mu_S([0,t])) =  \tau(P_1)\tau_{P_1\mathcal{M}P_1}(\mu^{P_1\mathcal{M}P_1}_{SP_1}([0,a])) + \tau(P_2)\tau_{P_2\mathcal{M}P_2}(\mu^{P_2\mathcal{M}P_2}_{SP_2}([0,b])) = \tau(P_1)a + \tau(P_2)b,\]
 and similarly, $t = \tau(P_1)c+\tau(P_2)d$. Together, with the assumption that $a>c$, this implies that $b < d$. We also have that $\tau(P_1)(a-c) = \tau(P_2)(d-b)$.  Next, it is easy to see that
 \[\operatorname{inf}(\{f^{P_1\mathcal{M}P_1}_{SP_1}(x) : x \in (c,a)\}) \geq \operatorname{sup}(\{f^{P_2\mathcal{M}P_2}_{SP_2}(x): x \in (b,d)\})\]

 We make a simple calculation,
\begin{eqnarray*}
&&\tau(P_1) \tau_{P_1\mathcal{M}P_1}(SP_1\mu_{SP_1}^{P_1\mathcal{M}P_1}([c,a])) - \tau(P_2) \tau_{P_2\mathcal{M}P_2}(SP_2\mu_{AP_2}^{P_2\mathcal{M}P_2}([b,d])) \\
&=& \tau(P_1)[F_{SP_1}^{P_1\mathcal{M}P_1}(a) - F_{SP_1}^{P_1\mathcal{M}P_1}(c)] +   \tau(P_2)[F_{SP_2}^{P_2\mathcal{M}P_2}(d) - F_{SP_2}^{P_2\mathcal{M}P_2}(b)]\\
&=& \tau(P_1)(a-c) \dfrac{\int_c^a f_{SP_1}^{P_1\mathcal{M}P_1}(x)dm(x)}{a-c} - \tau(P_2)(d-b)\frac{\int_b^d f_{SP_2}^{P_2\mathcal{M}P_2}(x)dm(x)}{d-b}\\
&\geq& 0
\end{eqnarray*}
Another simple calculation shows that
\begin {eqnarray*}
 F_S(t) - F_A(t) &=& \tau[S(\mu_{SP_1}^{P_1\mathcal{M}P_1}([0,a]) \oplus \mu_{SP_2}([0,b]))] - \tau[A(\mu_{AP_1}^{P_1\mathcal{M}P_1}([0,c])  \oplus \mu_{AP_2}([0,d]))] \\
&=& \tau(P_1) \tau_{P_1\mathcal{M}P_1}[SP_1\mu_{SP_1}^{P_1\mathcal{M}P_1}([0,a]) - AP_1\mu^{P_1\mathcal{M}P_1}_{AP_1}([0,c])]\\
& + &\tau(P_2) \tau_{P_2\mathcal{M}P_2}[SP_2\mu_{SP_2}^{P_2\mathcal{M}P_2}([0,b]) - AP_2\mu_{AP_2}^{P_2\mathcal{M}P_2}([0,d])]\\
&=& \tau(P_1) \tau_{P_1\mathcal{M}P_1}[SP_1\mu_{SP_1}^{P_1\mathcal{M}P_1}([0,c])  - AP_1\mu_{AP_1}^{P_1\mathcal{M}P_1}([0,c])]\\ 
&+& \tau(P_2) \tau_{P_2\mathcal{M}P_2}[SP_2\mu_{SP_2}^{P_2\mathcal{M}P_2}([0,d]) - AP_2\mu_{AP_2}^{P_2\mathcal{M}P_2}([0,d])]\\ 
&+& \tau(P_1) \tau_{P_1\mathcal{M}P_1}(SP_1\mu_{SP_1}^{P_1\mathcal{M}P_1}([c,a])) - \tau(P_2) \tau_{P_2\mathcal{M}P_2}(SP_2\mu_{AP_2}^{P_2\mathcal{M}P_2}([b,d])) \\
&\geq& \tau(P_1) L_{P_1 \mathcal{M} P_1}(AP_1,SP_1) + \tau(P_2) L_{P_2 \mathcal{M} P_2}(AP_2,SP_2)
\end {eqnarray*}
We conclude that 
\[L(A,S) \geq \sum_{1}^{2} \tau(P_m) L_{P_m \mathcal{M} P_m}(AP_m,SP_m)\]
For the last assertion, given the hypotheses, it is easy to see that 
\[f_A(t) = f_{AP_{m}}^{P_m \mathcal{M}P_m}\left(\dfrac{t-\tau(Q_{m-1})}{\tau(P_{m})}\right) \quad  \text{if } \tau(Q_{m-1}) \leq t  < \tau(Q_{m}), \quad 1 \leq m \leq k \]
and thus,
\[F_A(t) = \tau(AQ_{m-1}) + \tau(P_m) F_{AP_m}^{P_m\mathcal{M}P_m}\left(\dfrac{t-\tau(Q_{m-1})}{\tau(P_{m})}\right) \quad  \text{if } \tau(Q_{m-1}) \leq t  < \tau(Q_{m})  \]
And hence, 
\[\operatorname{inf}(\{F_S(t)-F_A(t) : t \in [\tau(Q_{m-1}),\tau(Q_m)]\}) = \tau((S-A)Q_{m-1}) + \tau(P_m)L_{P_m \mathcal{M} P_m}(AP_m,SP_m)\]
 The assertion follows. 

\end{proof}

We are also interested in Schur-Horn theorems in type $II_{\infty}$ factors. Approximate results in this setting were recently obtained by Argerami and Massey in\cite{ArgMas2I}. Let $\mathcal{M}$ be a $\sigma$ finite type $II_{\infty}$ factor and let $\tau$ be a faithful normal semifinite trace on $\mathcal{M}$. We will restrict our attention to masas $\mathcal{A}$ that admit a normal trace preserving conditional expectation.  We will refer such masas as atomic masas; Such masas are generated by their finite projections. Let $A$ be a positive trace class operator. Then, as in the case of positive operators in type $II_1$ factors, there exists a spectral scale $f_A$, this time on $[0,\infty)$ and a projection valued measure $\mu_A$, this time on $[0,\infty)$ so that 
\[\tau(A^{n}) =\int_0^{\infty} f_A(x)^n dm \quad \forall n, \quad A = \int_0^{\infty} f_A(t) d\mu_A(t) \]
Given, two trace class operators $A$ and $S$, we say that $S$ majorizes $A$, again written $A \prec S$ if inequalities analogous to (\ref{Maj21}) hold. For trace class operators, it is more natural to take the closure of the unitary orbit in the trace norm than in the operator norm; We thus define 
\begin{eqnarray}\label{TCO}
\mathcal{O}(S) = \overline{\{USU^{*} \, : U \in \mathcal{U}(\mathcal{M})\}}^{||\cdot||_1} 
\end{eqnarray}
when $S$ is trace class in a type $II_{\infty}$ factor. 

When the operators considered are not trace class, one needs to be more careful while considering majorization. As pointed out by Neumann\cite{NeuSH}, one needs to consider both the upper and lower spectral scales defined as 
\begin{eqnarray*}
U_A(x) &=& \inf\{t \mid \tau(E_{A}((t,\infty)) \leq x\}\\
L_A(x) &=& \sup\{t \mid \tau(E_{A}([0,t)) \leq x\} = -U_{-A}(x)
\end{eqnarray*}
When $A$ is trace class, $L_A$ becomes zero. For two positive operators $A$ and $S$, we say that $S$ majorizes $A$ if 
\begin{enumerate}\label{Maj2I}
\item We have the inequalities
 \begin{eqnarray}
\int_0^r U_A(x)dm \leq \int_0^r U_S(x)dm, \quad \int_0^r L_A(x)dm \geq \int_0^r L_S(x)dm, \,\, 0 \leq r < \infty
\end{eqnarray}
\item Additionally, if there is a $\lambda$ such that $S - \lambda I$ is trace class, then so is $A - \lambda I$ and $\tau(S - \lambda I) = \tau(A - \lambda I)$.  
\end{enumerate}

\section{A local Schur-Horn theorem}
Recall, see (\ref{EM}), that two positive operators $A$ and $S$ in a type $II_1$ factor $\mathcal{M}$ are said to be equimeasurable if $\tau(A^n) = \tau(S^n)$ for $n = 0, 1, \cdots$. This is equivalent to saying that the spectral measures and hence the spectral scales of $A$ and $S$ are identical. It is also routine to see that this is also equivalent to the existence of a sequence of unitary operators $\{U_n\}$ so that $||U_n S U_n^{*} - A|| \rightarrow 0$. 

An example of Popa\cite{PopInj} shows that equimeasurable operators need not be unitarily equivalent. Let $A$ lie inside a masa $\mathcal{A}$. The same example of Popa also shows that we cannot hope to even  ''locally'' conjugate $S$ into $A$, i.e, it is not possible to find a unitary $U$ and a projection $P$ in $\mathcal{A}$ so that $E(PUSU^{*}P) = AP$ and $A(I-P) \prec (I-P)USU^{*}(I-P)$. However, I show in proposition(\ref{prop1}) that this can be accomplished whenever $A \prec S$ but $A$ is not equimeasurable to $S$. 

	The proof of the Schur-Horn theorem has three steps. Given $A \in \mathcal{A}$ and $S \in \mathcal{M}$ with $A \prec S$, we 
\begin{enumerate}
\item First solve the problem ''locally''. That is, we find a projection $P$ in $\mathcal{A}$ and a unitary $U$ so that $E(PUSU^{*}P) = AP$ and so that $A(I-P) \prec (I-P)USU^{*}(I-P)$. This is done in proposition (\ref{prop1}).
\item Iterate the above procedure to make the projection $P$ as above as large as possible. We will end up with a projection $P$ in $\mathcal{A}$ and a unitary $U$ so that $E(PUSU^{*}P) = AP$ and $(I-P)USU^{*}(I-P) \approx A(I-P)$. This is done in proposition (\ref{SHM1}). This will yield us the first Schur-Horn theorem, the theorem that characterises the ''diagonals'' of a given hermitian operator.
\item Build upon this with some careful choices so that the projection $P$ actually equals $I$, yielding the Schur-Horn theorem. This is done in lemma (\ref{Eqmlem2}) and proposition (\ref{SH1Part}) in section $(5)$. 
\end{enumerate}

We start off with some elementary lemmas.

\begin{lemma}\label{2matlem}
Let $P$ be a projection of trace $\frac{1}{2}$ inside a masa $\mathcal{A}$ in a type $II_1$ factor $\mathcal{M}$. Let $A$ be a positive operator in $\mathcal{A}$ and let $S$ be a positive operator in $\mathcal{M}$ that commutes with $P$. With respect to the decomposition $I = P \oplus (I-P)$, we write(using an arbitrary partial isometry $V$ with $VV^{*}=P$ and $V^{*}V = I-P$ as the matrix unit $E_{12}$),
\[\mathcal{A} = \left( \begin{array}{cc}
\mathcal{A}_1 & 0\\
0 & \mathcal{A}_2 \end{array} \right) \quad A =   \left( \begin{array}{cc}
A_1 & 0\\
0 & A_2 \end{array} \right) \quad S =   \left( \begin{array}{cc}
S_1 & 0\\
0 & S_2 \end{array} \right) \]
where $\mathcal{A}_1$ and $\mathcal{A}_2$ are masas in $P\mathcal{M}P$, the operators $A_1$ and $A_2$ are in $\mathcal{A}_1$ and $S_1$ and $S_2$ are in $P\mathcal{M}P$.
 Assume that 
\begin{eqnarray}\label{dom}
\sigma(S_1) \geq \sigma(A_1) \geq \sigma(S_2) 
\end{eqnarray}

Then, there is a unitary $U$ so that 
\[E_{\mathcal{A}}(PUSU^{*}P) = AP \quad \text{i .e} \quad U \left( \begin{array}{cc}
S_1 & 0\\
0 & S_2 \end{array} \right) U^{*} =   \left( \begin{array}{cc}
X & \ast\\
\ast & Y \end{array} \right)\]
with $E_{\mathcal{A}_1}(X) = A_1$. We will automatically have that, \[\sigma_{(I-P)\mathcal{M}(I-P)}[(I-P)USU^{*}(I-P)] = \sigma_{P\mathcal{M}P}(Y) \subset \operatorname{conv}(\sigma(S))  = [\alpha(S_2),||S_1||]\]
\end{lemma}

\begin{proof}
 Let $H$ be the positive operator in $\mathcal{A}_1$ determined by the formula

\[H^2 = (A_1 - E_{\mathcal{A}_1}(S_2))( E_{\mathcal{A}_1}(S_1 - S_2))^{-1}\] 

 The operators $A_1$, $E_{\mathcal{A}_1}(S_1)$ and $ E_{\mathcal{A}_1}(S_2)$ form a commuting set and by the condition(\ref{dom}), it is easy to see that $H$ is a positive contraction. Now, let $U$ be the unitary given by
\[U =  \left( \begin{array}{cc}
H & \sqrt{I-H^2}\\
\sqrt{I - H^2} & H \end{array} \right)\]

Conjugating $S$ by $U$, we have that
\[U S U^{*} =  \left( \begin{array}{cc}
H S_1 H + \sqrt{I - H^2}  S_2  \sqrt{I - H^2} & \ast\\
\ast &  \sqrt{I-H^2} S_1 \sqrt{I-H^{2}}  +  H  S_2  H  \end{array} \right)\]
Another calculation shows that
\begin{eqnarray}
 \nonumber E_{\mathcal{A}_1}[H S_1 H + \sqrt{I - H^2}  S_2  \sqrt{I - H^2}] &=& H^2 E_{\mathcal{A}_1}(S_1) + (I-H^2)E_{\mathcal{A}_1}(S_2)\\
 \nonumber &=& H^2 (E_{\mathcal{A}_1}(S_1) - E_{\mathcal{A}_1}(S_2)) + E_{\mathcal{A}_1}(S_2)\\
 &=& A_1 
\end{eqnarray}

Recall that $\alpha(S_2)$ is the smallest point in the spectrum of $S_2$,
\begin{eqnarray}\alpha(S_2)I &=& (I-H^{2})\alpha(S_2) + H^{2} \alpha(S_2)\nonumber\\
& \leq& (I-H^{2})\alpha(S_1) + H^{2} \alpha(S_2) \label{tech1}\\
&\leq& \sqrt{I-H^2} S_1 \sqrt{I-H^{2}}  +  H  S_2  H \nonumber\\
& \leq& (I-H^{2})||S_1|| + H^{2} ||S_1|| \label{tech2}\\
&\leq& ||S_1|| I \nonumber 
\end{eqnarray}

We conclude that 
\begin{eqnarray}
\sigma_{P\mathcal{M}P}(Y) = \sigma_{P\mathcal{M}P}[\sqrt{I-H^2} S_1 \sqrt{I-H^{2}}  +  H  S_2  H]  \subset [\alpha(S_2),||S_1||] = \operatorname{conv}(\sigma(S))
\end{eqnarray}

\end{proof}

\begin{remark}\label{greater}
In the setup of lemma (\ref{2matlem}), suppose we have that $\sigma(S_1) > \sigma(A_1) \geq \sigma(S_2)$, we will have that \[\sigma_{P\mathcal{M}P}[(I-P)USU^{*}(I-P)] \subset (\alpha(S_2),||S_1||)\]
This is because, in lines (\ref{tech1}) and (\ref{tech2}), we will have strict inequality instead of mere inequality. 
\end{remark}

We now show that majorization is preserved under small perturbations. 

\begin{lemma}\label{lem1}
Let $A\prec S$ be positive operators in a type $II_1$ factor $\mathcal{M}$ and suppose
\[A =  \left( \begin{array}{ccc}
A_1 & 0 & 0\\
0 & A_2 & 0\\
0 & 0 & A_3\end{array} \right), \quad S =   \left( \begin{array}{ccc}
S_1 & 0 & 0\\
0 & S_2 & 0\\
0 & 0 & S_3\end{array} \right)\]
where the decomposition is with respect to $I = P \oplus Q \oplus R$ where $P, Q$ are orthogonal projections commuting with $A$ and $S$ and $R = I-P-Q$. Suppose that 
\[\sigma_{P\mathcal{M}P}(A_1) \geq \sigma_{Q\mathcal{M}Q}(A_2) \geq  \sigma_{R\mathcal{M}R}(A_3), \quad \sigma_{P\mathcal{M}P}(S_1)  \geq \sigma_{Q\mathcal{M}Q}(S_2) \geq \sigma_{R\mathcal{M}R}(S_3)\]
  Let $T$ be a positive operator in $Q\mathcal{M}Q$ with the same trace as $S_2$ and so that 
  \[\sigma_{P\mathcal{M}P}(S_1)  \geq \sigma_{Q\mathcal{M}Q}(T) \geq \sigma_{R\mathcal{M}R}(S_3)\]
  Suppose, $\tau_{P\mathcal{M}P}(S_1 - A_1) \geq -L_{Q\mathcal{M}Q}(A_2,T)\dfrac{\tau(Q)}{\tau(P)}$, (which holds in particular if $\tau(P(S - A))\geq \tau(Q)\tau_{Q\mathcal{M}Q}(T)$). Then,  we have the majorization relation,
\[ \left( \begin{array}{ccc}
A_1 & 0 & 0\\
0 & A_2 & 0\\
0 & 0 & A_3\end{array} \right) \prec   \tilde{T} :=  \left( \begin{array}{ccc}
S_1 & 0 & 0\\
0 & T & 0\\
0 & 0 & S_3\end{array} \right)\]
If in addition, we have that $F_S - F_A$ is strictly positive on $(0,1)$, then $F_{\tilde{T}}-F_A$ is strictly positive on $(0,1)$ as well. 
\end{lemma}
\begin{proof}
It is easy to see that we have
\[F_R(x) -F_A(x)=  F_S(x) - F_A(x), \quad  x \in [0,\tau(P)] \cup [\tau(P+Q),1]\]
which is non-negative by hypothesis on $[0,1]$ and by assertions $(1)$ and $(3)$ of (\ref{Llemma}) that for $x \in [\tau(P),\tau(P+Q)]$
\begin{eqnarray*}
F_R(x)-F_A(x) &\geq& F_R(\tau(P))-F_A(\tau(P))  + \tau(Q)L_{Q\mathcal{M}Q}(A_2,T) \\
&=& \tau(P)\tau_{P\mathcal{M}P}(S_1 - A_1)+\tau(Q)L_{Q\mathcal{M}Q}(A_2,T)\\
&\geq& 0
\end{eqnarray*}
The assertion follows. The case when $F_S - F_A$ is strictly positive on $(0,1)$ is treated similarly. 
 \end{proof}

The following proposition is the critical step in the proof of the first Schur-Horn theorem, namely theorem (\ref{conj1}). It shows that the problem can be ''locally'' solved. Precisely,
 \begin{proposition}\label{prop1}
 Let $\mathcal{A}$ be a masa in a type $II_1$ factor $\mathcal{M}$. Let  $A \in \mathcal{A}$ and $S \in \mathcal{M}$ be positive operators with $A \prec S$. Assume that $A \notin \mathcal{O}(S)$. Then, there are projections $P$ in $\mathcal{A}$ and $Q$ in $\mathcal{M}$ with $P \leq Q$ and $\tau(Q) \leq 4\tau(P)$ and a unitary $U$ in $\mathcal{M}$ satisfying $U-I=Q(U-I) $  so that  
\[E(PUSU^{*}P) = AP, \qquad A(I-P) \prec (I-P)USU^{*}(I-P) \]
 \end{proposition}

\begin{proof}
Let us assume that $A$ and $S$ are contractions; It is easy to see that proving the proposition for contractions will yield the general result, by scaling. Let $f_A, f_S$ be the spectral scales of $A, S$ respectively and let $F_A, F_S$ be the Ky Fan norm functions associated to $A$ and $S$ respectively, see (\ref{defF}).

Since $A \prec S$, we have that
\begin{eqnarray}\label{MajDef}
 \int_0^r f_A(x)dm \leq \int_0^r f_S(x)dm, \,\,\, 0 \leq r \leq 1 \quad \text{and} \quad \int_0^1 
f_A(x)dm = \int_0^1 f_S(x)dm
\end{eqnarray}

Assume for now that $f_A \neq f_S$ almost everywhere and that $F_A(x) < F_S(x)$ on $(0,1)$. Once we have proved the proposition under this assumption, the general case will follow using routine arguments - See the last paragraph of the proof. 

Let $I = \{x \in [0,1]\, \mid f_A(x) < f_S(x)\}$ and let $J = \{x \in [0,1]\, \mid f_A(x) > f_S(x)\}$. Pick $0 < a  < 1$ so that $I \cap [a-\epsilon,a]$ and $J \cap [a,a+\epsilon]$ have positive Lebesgue measure for every $\epsilon > 0$. This can be done as follows: Recall that the functions $F_S$ and $F_A$ are continuous. Let $a$ be a number such that $F_S(a)-F_A(a) = \operatorname{max}(\{F_S(x) - F_A(x) \mid x \in [0,1]\}$. The desired property is now easy to verify. Suppose $I \cap [a-\epsilon,a]$ has zero measure, then
 \[[F_S(a) - F_A(a)] - [F_S(a-\epsilon) - F_A(a-\epsilon)] = \int_{a-\epsilon}^{a} (f_S(x)-f_A(x))dm < 0,\] contradicting the choice of $a$. The verification of the other desired property is similar. 

 Next, choose numbers $b,c$ with $0 < b < a < c < 1$ and define the number $\alpha$ by
\[\alpha:= \operatorname{inf}_{b \leq x \leq c} F_S(x) - F_A(x) \]
Since $F_A$ and $F_S$ are continuous and $F_S - F_A$ is strictly positive on $(0,1)$, we have that $\alpha > 0$.

Pick $\epsilon <  \dfrac{\alpha}{2}$ and pick subsets $X$ and $Y$ of positive measure in $I \cap [a-\epsilon,a]$ and $J \cap [a,a + \epsilon]$. Let $L_1, L_2, L_3, L_4$ be the sets 
\[L_1 = \{f_S(x) \mid x \in X\}, \,  L_2 = \{f_A(x) \mid x \in X\}, \,  L_3 = \{f_A(x) \mid x \in Y\}, \,  L_4 = \{f_S(x) \mid x \in Y\}.\] 

We may further arrange, by passing to subsets, if needed, that the following are satisfied:
\begin{enumerate}\label{ordercon}
 \item $m(X) = m(Y)$.
 \item $L_1 > L_2 > L_3 > L_4$
 \end{enumerate}
 
 For the second assertion, we use that $f_A, f_S$ are right continuous and non-increasing. See the figure below for a schematic description: 

\begin{figure}[H]

\centering

\includegraphics[width=80mm]{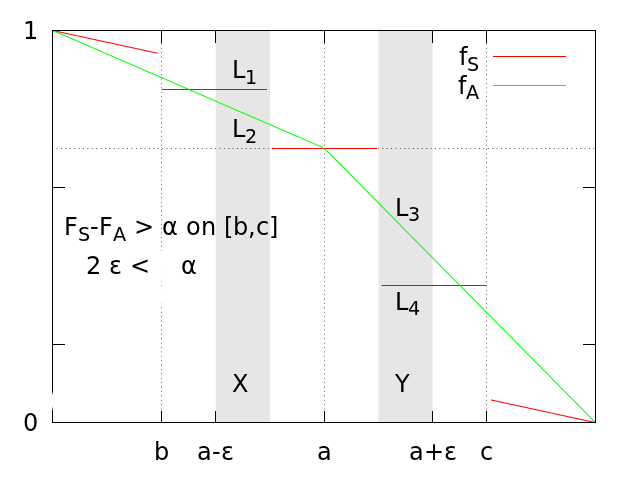}
\caption{Illustrating proposition (\ref{prop1})}
\end{figure}
 
Let $Q_1$ be the projection $\mu_A(X) \oplus \mu_A(Y)$ and $Q_2$ the projection $\mu_S(X) \oplus \mu_S(Y)$ and let $Q = Q_1 \vee Q_2$. Choose a unitary $V$ such that
\[V\mu_S(X)V^{*} = \mu_A(X), \qquad V\mu_S(Y)V^{*} = \mu_A(Y)\]
It is easy to see that the unitary $V$ may be chosen so that $V = QVQ \oplus (I-Q)$.
With respect to the decomposition $Q_1 = \mu_A(X) \oplus \mu_A(Y)$, we may write 
\[Q_1AQ_1 = \left( \begin{array}{cc}
A_1 & 0\\
0 & A_2 \end{array} \right) , \qquad Q_1VSV^{*}Q_1 = \left( \begin{array}{cc}
S_1 & 0\\
0 & S_2 \end{array} \right) \]
We have that(suppressing the explicit identification of the ambient algebra), by part $(ii)$ of (\ref{ordercon}) that, 
 \[\sigma(S_1) > \sigma(A_1) > \sigma(A_2) > \sigma(S_2)\]
 By lemma(1), there is a unitary $W$ satisfying  $W = Q_1WQ_1 \oplus (I-Q_1)$ inside $\mathcal{M}$ so that 
\begin{eqnarray}\label{local}
Q_1W VSV^{*} W^{*} Q_1 = \left( \begin{array}{cc}
Z_1 & \ast\\
\ast & Z_2 \end{array} \right) 
\end{eqnarray}
where $E_{\mathcal{A}\mu_A(X)}(Z_1) = A_1$ and further, $\sigma(Z_2) \in \operatorname{conv}[\sigma(S_1)\cup \sigma(S_2)]$. Let $U$ be the unitary $WV$ and $T$ the operator $T = USU^{*}$. We note that 
\[U = QUQ \oplus (I-Q)\qquad \text{or} \quad U-I=Q(U-I)\]
Letting $P = \mu_A(X)$, the equation (\ref{local}) implies that 
\begin{eqnarray}\label{part1}
E_{\mathcal{A}}(PTP) = AP, \end{eqnarray}
We have the trace inequality 
\begin{eqnarray}\label{traceP}
\tau(Q) = \tau(Q_1 \vee Q_2) \leq 2\tau(Q) = 4 \tau(P)
\end{eqnarray}

We now show that we also have the other required majorization condition, namely
\[A(I-P) \prec (I-P)T(I-P)\]
Let us decompose $A$ and $T$ using the projection decompositions

 \[I = \mu_{A}([0,a-\epsilon]) \oplus  \mu_A(X) \oplus [\mu_{A}([a-\epsilon,a+\epsilon]) - \mu_A(X) - \mu_A(Y)]  \oplus \mu_A(Y) \oplus \mu_{A}([a+\epsilon,1])\] 
 and
  \[I = \mu_{T}([0,a-\epsilon]) \oplus  \mu_A(X) \oplus [\mu_{T}([a-\epsilon,a+\epsilon]) - \mu_A(X) - \mu_A(Y)]  \oplus \mu_A(Y) \oplus \mu_{T}([a+\epsilon,1])\]
  respectively. (Here we use that $\mu_A(X)$ and $\mu_A(Y)$ are sub projections of $\mu_T([a-\epsilon,a+\epsilon])$ because of the special form of the averaging unitary $U$),   

\[ A =  \left( \begin{array}{ccccc}
B_1 & 0 & 0 & 0 & 0\\
0 & A_1 & 0 & 0 & 0\\
0 & 0 & A_3 & 0 & 0\\
0 & 0 & 0 & A_2 & 0\\
0 & 0 & 0 & 0 & B_2\end{array} \right), \quad  T=  \left( \begin{array}{ccccc}
T_1 & 0 & 0 & 0 & 0\\
0 & Z_1 & 0 & \ast & 0\\
0 & 0 & Z_3 & 0 & 0\\
0 & \ast & 0 & Z_2 & 0\\
0 & 0 & 0 & 0 & T_2\end{array} \right)\]   

 We have that $E_{\mathcal{A}P}(Z_1) = A_1$ and (we omit writing down the ambient algebra explicitly),
\begin{eqnarray}\label{hyp1}
\sigma(B_1) \geq \sigma(A_1) \cup \sigma(A_2) \cup \sigma(A_3) \geq \sigma(B_2), \,\, \sigma(T_1) \geq \sigma(Z_1) \cup \sigma(Z_2) \cup \sigma(Z_3) \geq \sigma(T_2)
\end{eqnarray}
and further, 
\begin{eqnarray}\label{hyp2}
 \tau(T_1 - B_1) = F_S(a-\epsilon)-F_A(a-\epsilon) \geq \alpha > 2 \epsilon  >  \tau[\mu_A([a-\epsilon,a+\epsilon])]
\end{eqnarray}

Let $\tilde{E}$ be the conditional expectation given by compression to the block diagonal followed by applying $E$ to the second diagonal entry. We have that 
\[ \tilde{E}(T) = R =  \left( \begin{array}{ccccc}
T_1 & 0 & 0 & 0 & 0\\
0 & A_1 & 0 & 0 & 0\\
0 & 0 & Y & 0 & 0\\
0 & 0 & 0 & S_3 & 0\\
0 & 0 & 0 & 0 & T_2\end{array} \right) \, \prec \, T \]

By the calculations (\ref{hyp1}) and (\ref{hyp2}), the operators $A$ and $R$ satisfy the hypothesis of lemma(\ref{lem1}) (note here that $S$ and hence, $T$, is a contraction) and thus $A \prec R$, namely

\[   \left( \begin{array}{ccccc}
B_1 & 0 & 0 & 0 & 0\\
0 & A_1 & 0 & 0 & 0\\
0 & 0 & A_2 & 0 & 0\\
0 & 0 & 0 & A_3 & 0\\
0 & 0 & 0 & 0 & B_2\end{array} \right) \prec  \left( \begin{array}{ccccc}
T_1 & 0 & 0 & 0 & 0\\
0 & A_1 & 0 & 0 & 0\\
0 & 0 & Y & 0 & 0\\
0 & 0 & 0 & S_3 & 0\\
0 & 0 & 0 & 0 & T\end{array} \right)\]   

This implies that
\[   \left( \begin{array}{cccc}
B_1 & 0 & 0 & 0\\
0 & A_2 & 0 & 0\\
0 & 0 & A_3 & 0\\
0 & 0 & 0 & B_2\end{array} \right) \prec   \left( \begin{array}{cccc}
T_1 & 0 & 0 & 0 \\
0 & B & 0 & 0\\
0 & 0 & S_3 & 0\\
0 & 0 & 0 & T_2\end{array} \right)\]
or, in other words,
\[(I-P)A \prec (I-P) T(I-P)\] 

This completes the proof when $f_A \neq f_S$ a.e. and $F_A < F_S$ on $(0,1)$. Now, we look at the general case, dropping the assumption that $f_A \neq f_S$ almost everywhere and that $F_A(x) < F_S(x)$ on $(0,1)$. Define $X: \{x \in [0,1] \mid f_A(x) = f_S(x)\}$, which may now have positive measure. Pick a unitary $U$ that conjugates $\mu_S(X)$ onto $\mu_A(X)$. We may write, under the decomposition $I = \mu_A(X) \oplus \mu_A(X^c) = P \oplus Q$, 
\[A = A_1 \oplus A_2 \quad \text{ and } USU^{*} = S_1  \oplus S_2\]
where $A_1$ and $S_1$ have the same spectral measure, hence $A_1 \in \mathcal{O}(S_1)$ by \cite{ArvKad}[Theorem 5.4] and $A_2 \prec S_2$ with the property that the spectral scales $f_{A_2}^{Q \mathcal{M} Q}$ and $f_{S_2}^{Q \mathcal{M} Q}$  satisfy $f_{A_2}^{Q \mathcal{M} Q} \neq f_{S_2}^{Q \mathcal{M} Q}$ almost everywhere. Since $A \notin \mathcal{O}(S)$, $A_2$ and $S_2$ are non-zero.

 Since the Ky Fan norm functions are continuous, we may find two points $\{a_1, a_2\}$ so that $F_{A_2}^{Q \mathcal{M} Q}(a_i) = F_{S_2}^{Q \mathcal{M} Q}(a_i)$ and $F_{A_2}^{Q \mathcal{M} Q}(x) < F_{S_2}^{Q \mathcal{M} Q}(x)$  for  $x \in (a_1,a_2)$. Pick a unitary $V$ that conjugates $\mu_{S_2}([a_1,a_2])$ onto $\mu_{A_2}([a_1,a_2])$ and commutes with $\mu_A(X)$. We may write, under the decomposition $I = \mu_A(X) \oplus \mu_{A_2}([a_1,a_2]) \oplus [I- \mu_A(X) - \mu_{A_2}([a_1,a_2])]$,
 \[A = A_1 \oplus A_3 \oplus A_4 \quad \text{ and } VUSU^{*}V^{*} = S_1  \oplus S_3 \oplus S_4\]
where $A_3 \prec S_3$ and whose respective spectral scales are non-equal a.e. Further, the Ky Fan norm functions satisfy $F_{A_3} < F_{S_3}$ on $(0,1)$. Note that we also have that $A_4 \prec S_4$. 

The proposition now applies to $(A_3, S_3)$ and yields the desired conclusion for $A$ and $VUSU^{*}V^{*}$.

\end{proof}

Given positive operators $A \in \mathcal{A}$ and $S \in \mathcal{M}$ as above, we say that $(U, P)$ is a \emph{partial solution} if $U$ is a unitary, $P$ is a projection in  $\mathcal{A}$, $E(PUSU^{*}P)=AP$ and $A(I-P)\prec (I-P)USU^{*}(I-P)$. With this notation, we have the following corollary,
\begin{corollary}\label{cor1}
 Let $\mathcal{A}$ be a masa in a type $II_1$ factor $\mathcal{M}$. Let  $A \in \mathcal{A}$ and $S \in \mathcal{M}$ be positive operators with $A \prec S$ and let $(U,P)$ be a partial solution. Assume that $A(I-P) \not\approx (I-P)USU^{*}(I-P)$ inside $(I-P)\mathcal{M}(I-P)$. Then, there are projections $P_1 \in \mathcal{A}$ and $Q_1 \in \mathcal{M}$ so that $P < P_1 < Q_1$ and $\tau(Q_1 - P_1) \leq 4\tau(P_1 - P)$ and a unitary $V$ in $\mathcal{M}$ so that $(V,P_1)$ is a partial solution and 
\begin{eqnarray}\label{Pres}
V-U= (Q_1-P)(V-U)
\end{eqnarray}

\end{corollary}
\begin{proof}
Apply the previous proposition (\ref{prop1}) to $A(I-P)$ and $(I-P)USU^{*}(I-P)$ inside $(I-P)\mathcal{M}(I-P)$ to get a unitary $W$ in $(I-P)\mathcal{M}(I-P)$ with $W-I_{ (I-P)\mathcal{M}(I-P)} = Q_1(W-I_{ (I-P)\mathcal{M}(I-P)})$, a projection $Q_0$ in $(I-P)\mathcal{M}(I-P)$ and a projection $P_0$ in $A(I-P)$ so that (inside $(I-P)\mathcal{M}(I-P)$) we have 
\[E(P_0(I-P)USU^{*}(I-P)P_0) = A(I-P)P_0, \quad A(I-P)(I-P_0) \prec (I-P_0)(I-P)USU^{*}(I-P)(I-P_0) \] and also,
\[\tau_{(I-P)\mathcal{M}(I-P)}(Q_0)  \leq 4\tau_{(I-P)\mathcal{M}(I-P)}(P_0)\] 
Now, let $V = (P \oplus W)U$, let $P_1 = P \oplus P_0$ and let $Q_1=P\oplus Q_0$. Here, we interpret $W, P_0, Q_0$ which are operators in $(I-P)\mathcal{M}(I-P)$, in the natural fashion inside $\mathcal{M}$. It is clear that $V$ is a unitary; we see that 
\[\tau(Q_1 - P_1) = \tau(I-P)\tau_{(I-P)\mathcal{M}(I-P)}(Q_0) \leq 4\tau(I-P)\tau_{(I-P)\mathcal{M}(I-P)}(P_0)  \leq 4\tau(P_1 - P)\] and that
\[E(P_1VSV^{*}P_1)=AP_1, \quad A(I-P_1)\prec (I-P_1)USU^{*}(I-P_1)\]
Further, 
\[V-U = ((P\oplus W)-I)U = (W-(I-P))U = Q_0(W-(I-P))U = (Q_1-P)(V-U).\] We conclude that $(V,P_1)$ is a partial solution with the desired properties. 


\end{proof}

\section{Diagonals of positive operators in type $II_1$ factors}

We deduce a Schur-Horn theorem in type $II_1$ factors from corollary(\ref{cor1}) using an induction argument.

 \begin{theorem}\label{SHM1}
Let $\mathcal{A}$ be a masa in a type $II_1$ factor $\mathcal{M}$. If $A \in \mathcal{A}$ and $S \in \mathcal{M}$ are positive operators with $A \prec S$. Then, there is a unitary $U$ and a projection $P$ in $\mathcal{A}$ such that
\[E(PUSU^{*}P) = AP \quad \text{ and } \quad (I-P)USU^{*}(I-P) \approx A(I-P).\]

\end{theorem}

\begin{proof}

 If $A \in \mathcal{O}(S)$, there is nothing to prove; Just set $P$ to be zero and the unitary to be the identity. Let us therefore assume that $A \notin \mathcal{O}(S)$.
  
 Let $A \in  \mathcal{A}$ and $S \in \mathcal{M}$ be positive operators so that $A \prec S$. Let $\mathcal{X}$ be the collection of all tuples $(U,P)$ where $U$ is a unitary in $\mathcal{M}$ and $P$ is a projection in $\mathcal{A}$ with $E(PUSU^{*}P) = AP$ and $A(I-P) \prec (I-P)USU^{*}(I-P)$. Define an ordering $\leq$ on the set $\mathcal{X}$ by $(U_1,P_1) \leq (U_2,P_2)$ if
\begin{enumerate}
 \item $P_1 \leq P_2$, i.e $P_1 P_2 = P_1$,
 \item There is a projection $Q$, with $Q > P_2$ and satisfying $\tau(Q-P_2) \leq 4\tau(P_2-P_1)$ so that $U_2 - U_1 = (Q-P_1)(U_2-U_1)$. 
\end{enumerate}
 The set $\mathcal{X}$ with the given ordering is a poset. To see this, suppose $(U_1,P_1) \leq (U_2,P_2)$ and $(U_2,P_2) \leq (U_3,P_3)$; While showing that $(U_1,P_1) \leq (U_3,P_3)$, property $(1)$ is immediate. Let $Q_1, Q_2$ be the projections that ensure condition $(3)$ in the inequalities $(U_1,P_1) \leq (U_2,P_2)$ and $(U_2,P_2) \leq (U_3,P_3)$ respectively. Take $Q_3 = Q_1 \vee Q_2$. By definition, we have that $U_3-U_2 = (Q_2-P_2)(U_3-U_2)$ and  $U_2-U_1 = (Q_1-P_1)(U_2-U_1)$. This yields that $(I-Q_2+P_1)(U_2-U_1) = 0$ and since $I-Q_2$ is orthogonal to $P_1$, we have $(I-Q_2)(U_3-U_2) = P_2(U_3-U_2)= 0$. Similarly, $(I-Q_1)(U_2-U_1) = P_1(U_2-U_1)= 0$. We see that 
 \begin{eqnarray*}
 (I-Q_3)(U_3-U_1) &=& (I-Q_3)(U_3-U_2) + (I-Q_3)(U_2-U_1)\\
 &=&(I-Q_3)(I-Q_2)(U_3-U_2)+(I-Q_3)(I-Q_1)(U_2-U_1) \\
 &=&0
 \end{eqnarray*}
And further,
 \begin{eqnarray*}
 P_1(U_3 - U_1) &=& P_1(U_3 - U_2) + P_1(U_2 - U_1)\\
  &=& P_1P_2(U_3-U_2) + P_1(U_2-U_1) \\
  &=& 0
   \end{eqnarray*}
   Thus, $(I-Q_3+P_1)(U_3-U_1) = 0$, giving us that
   \begin{eqnarray}\label{comp1}
   U_3-U_1 = (Q_3-P_1)(U_3-U_1)
   \end{eqnarray}
 Next, recalling that $\tau(Q_1-P_2)\leq 4\tau(P_2-P_1)$ and $\tau(Q_2-P_3)\leq4\tau(P_3-P_2)$ and that $Q_1\wedge Q_2 \geq P_2$,
 \begin{eqnarray}\label{ind}
\nonumber \tau(Q_3 - P_3) &=& \tau(Q_1 \vee Q_2) - \tau(P_3)\\
\nonumber &=&   \tau(Q_1) + \tau(Q_2) - \tau(Q_1 \wedge Q_2) - \tau(P_3)\\
\nonumber &\leq& 5\tau(P_2) - 4\tau(P_1) + 5\tau(P_3) - 4\tau(P_2) - \tau(P_2) - \tau(P_3)\\
&=& 4\tau(P_3 - P_1)  
\end{eqnarray}
In the third line, we used the fact that $Q_1 \wedge Q_2 \geq P_2$. We conclude from (\ref{comp1}) and (\ref{ind}) that $(\mathcal{X},\leq)$ is a poset. 

In what follows, we consider $\mathcal{M}$ in its standard form, sitting inside $L^2(\mathcal{M},\tau)$. Let $\{(U_{\alpha},P_{\alpha})\}_{\alpha \in I}$ be a chain in $\mathcal{X}$. Since the projections $P_{\alpha}$ are increasing, they have a strong operator limit, which we denote by $P$. Fix operators $T$ and $S$ in $\mathcal{M}$. We claim that  $\operatorname{lim}_{\alpha}<U_{\alpha}T\Omega,S\Omega>$ exists. Fix $\epsilon > 0$. Since $P_{\alpha}$ converge in the SOT, there is an $\alpha$ so that if $\beta > \alpha$, then  $\tau(P_{\beta}-P_{\alpha}) < \epsilon$. Let $Q$ be the projection that witnesses $(U_{\alpha},P_{\alpha}) < (U_{\beta},P_{\beta})$, i.e. we have that 
\[P_{\beta} < Q, \qquad \tau(Q-P_{\beta}) \leq 4\tau(P_{\beta}-P_{\alpha}), \qquad  U_{\beta} - U_{\alpha} = (Q-P_{\alpha})(U_{\beta}-U_{\alpha}). \]
We have that
\begin{eqnarray*}
|<U_{\beta}T\Omega,S\Omega>-<U_{\alpha}T\Omega,S\Omega>|&=&|<(U_{\beta}-U_{\alpha})T\Omega,(Q-P_{\alpha})S\Omega>|\\
&=&|\tau((Q-P_{\alpha})(U_{\beta}-U_{\alpha})TS^{*})|\\
&\leq& ||U_{\beta}-U_{\alpha}||||T||||S^{*}||\tau(Q-P_{\alpha})\\
&\leq&10||T||||S||\tau(P_{\beta}-P_{\alpha})\\
&\leq& 10\epsilon ||T||||S||
\end{eqnarray*}
Define the sesquilinear forms on $L^2(\mathcal{M},\tau) \times L^2(\mathcal{M},\tau)$,
\[\phi_{\alpha}(\xi,\eta) = <U_{\alpha}\xi,\eta> \]
These converge pointwise on $\mathcal{M}\Omega \times \mathcal{M}\Omega$. Denote the limit by $\phi(\cdot,\cdot)$. Then, it is easy to see that $|\phi(T\Omega,S\Omega)| \leq ||T\Omega||_2||S\Omega||_2$ and this shows that $\phi$ is extendable as a sesquilinear form to $L^2(\mathcal{M},\tau) \times L^2(\mathcal{M},\tau)$, with $|\phi(\xi,\eta)| \leq ||\xi||||\eta||$ for all $\xi, \eta \in L^2(\mathcal{M},\tau)$ . By the Riesz representation theorem there is a contraction $U$ in $\mathcal{B}(L^{2}(\mathcal{M},\tau))$ such that 
\[ <U\xi,\eta> = \phi(\xi,\eta) = \operatorname{lim}_{\alpha}<U_{\alpha}\xi,\eta>\]
This means in particular that $U_{\alpha}$ converges to $U$ in the WOT and thus, $U$ is in $\mathcal{M}$. We now show that we actually have SOT convergence - This will imply that $U$ is in fact a unitary. 

Since the operators are bounded, it is enough to check for SOT convergence on the dense set $\mathcal{M}\Omega$. Let $T$ be in $\mathcal{M}$. Since $U_{\beta}$ converges in the WOT to $U$, we have that
\[||(U-U_{\alpha})T\Omega||_{2} \leq \operatorname{lim inf}_{\beta}||(U_{\beta}-U_{\alpha})T\Omega||_{2} \]
For any $\alpha < \beta$, let $Q_{\alpha}^{\beta}$ be the projection that witnesses $(U_{\alpha},P_{\alpha}) < (U_{\beta},P_{\beta})$. We have that,
\begin{eqnarray*}
||(U-U_{\alpha})T\Omega||_{2} &\leq& \operatorname{lim inf}_{\beta}||(U_{\beta}-U_{\alpha})T\Omega||_{2}\\
&=&\operatorname{lim inf}_{\beta}||(Q_{\alpha}^{\beta}-P_{\alpha})(U_{\beta}-U_{\alpha})T\Omega||_{2}\\
&=&\operatorname{lim inf}_{\beta}||\tau((Q_{\alpha}^{\beta}-P_{\alpha})(U_{\beta}-U_{\alpha})TT^{*}(U_{\beta}-U_{\alpha})^{*})\\
&\leq&\operatorname{lim inf}_{\beta}\tau(Q_{\alpha}^{\beta}-P_{\alpha})||T||^{2}||U_{\beta}-U_{\alpha}||^{2}\\
&\leq&\operatorname{lim inf}_{\beta} 20\tau(P_{\beta}-P_{\alpha})||T||^{2}\\
&\leq& 20\tau(P-P_{\alpha})||T||^{2}
\end{eqnarray*}
We used the fact that $\tau(Q_{\alpha}^{\beta}-P_{\alpha}) \leq 4\tau(P_{\beta}-P_{\alpha})$ and hence, that $\tau(Q_{\alpha}^{\beta}-P_{\alpha}) \leq 5\tau(P_{\beta}-P_{\alpha})$ in line $(5)$. 

It follows that $U_{\alpha}$ converges to $U$ in the SOT. A similar calculation shows that $U_{\alpha}^{*}$ converges in the SOT to $U^{*}$. Since the $U_{\alpha}$ are uniformly bounded in norm(by $1$), we have that $U_{\alpha}^{*}U_{\alpha}$ converges in the SOT to $U^{*}U$ and thus $U^{*}U = I$. We conclude that $U$ is a unitary. 

The strong $*$ convergence of the $U_{\alpha}$ to $U$ implies that the automorphisms $\operatorname{Ad}(U_{\alpha})$ converge in the point $2$ norm topology to $\operatorname{Ad}(U)$. Now, we have that
\[E(PUSU^{*}P) = \operatorname{lim}_{SOT} E(P_{\alpha}U_{\alpha}SU_{\alpha}^{*}P_{\alpha}) = \operatorname{lim}_{SOT} AP_{\alpha} = AP\] 

Concerning majorization, $ A(I-P_{\alpha}) \prec (I-P_{\alpha})U_{\alpha}SU_{\alpha}^{*}(I-P_{\alpha}) $ for every $\alpha$ and hence, passing to the strong operator limit,  \[A(I-P) \prec (I-P)USU^{*}(I-P).\] 

We conclude that $(U,P)$ is in $\mathcal{X}$. We now show that for every $\alpha$, we have that $(U_{\alpha},P_{\alpha}) < (U,P)$. Pick a sequence $\alpha_n$, $n = 1,2, \cdots$ in $I$ with $P_{\alpha_1}=P_{\alpha}$ so that $P_{\alpha_n} $ is increasing and converges to $P$ in the SOT. As above, let $Q_{\alpha_n}^{\alpha_m}$ for $n < m$ be the projection that witnesses $(U_{\alpha_n},P_{\alpha_n}) < (U_{\alpha_m},P_{\alpha_m})$. Let $Q$ be the projection 
\[Q := \vee_{n} Q_{\alpha_n}^{\alpha_{n+1}}\]

For any $N$, we have that 
\[\tau(\vee_{1}^{N} Q_{\alpha_n}^{\alpha_{n+1}}) = \tau(Q_{\alpha_N}^{\alpha_{N+1}}) + \tau(\vee_{1}^{N-1} Q_{\alpha_n}^{\alpha_{n+1}}) - \tau(Q_{\alpha_N}^{\alpha_{N+1}} \wedge\vee_{1}^{N-1} Q_{\alpha_n}^{\alpha_{n+1}})\]
Since $Q_{\alpha_m}^{\alpha_{m+1}}$ is larger than $P_{\alpha_{m+1}}$ and a fortiori larger than $P_{\alpha_{k}}$ for $k \leq m+1$, we have that  $P_{\alpha_N} \leq Q_{\alpha_N}^{\alpha_{N+1}} \wedge\vee_{1}^{N-1} Q_{\alpha_n}^{\alpha_{n+1}}$. Further, $\tau(Q_{\alpha_N}^{\alpha_{N+1}}) \leq 5\tau(P_{\alpha_{N+1}})-4\tau(P_{\alpha_N})$. As a result,

\begin{eqnarray*}
\tau(\vee_{1}^{N} Q_{\alpha_n}^{\alpha_{n+1}} - P_{\alpha_{N+1}})  &\leq& 5\tau(P_{\alpha_{N+1}})-4\tau(P_{\alpha_N}) + \tau(\vee_{1}^{N-1} Q_{\alpha_n}^{\alpha_{n+1}}) - \tau(P_{\alpha_N}) - \tau(P_{\alpha_{N+1}})\\
&=&4(\tau(P_{\alpha_{N+1}})-\tau(P_{\alpha_N})) +  \tau(\vee_{1}^{N-1} Q_{\alpha_n}^{\alpha_{n+1}} - P_{\alpha_{N}})
\end{eqnarray*}
The sum telescopes to yield
\[\tau(\vee_{1}^{N} Q_{\alpha_n}^{\alpha_{n+1}} - P_{\alpha_{N+1}})  \leq 4(\tau(P_{\alpha_{N+1}})-\tau(P_{\alpha_1})) = 4(\tau(P_{\alpha_{N+1}})-\tau(P_{\alpha}))\] 
Taking the limit as $N \rightarrow \infty$, we get that
\[\tau(Q-P) \leq 4\tau(P-P_{\alpha})\]

We conclude that $(U_{\alpha},P_{\alpha}) < (U,P)$ for every $\alpha$. Thus, every chain has an upper bound and now, Zorn's lemma gives us that there is a maximal element in $\mathcal{X}$. Let $(U,P)$ be this maximal element. If $A(I-P)$ is not equimeasurable to $(I-P)USU^{*}(I-P)$, corollary (\ref{cor1}) applies and yields us a larger element in $\mathcal{X}$, yielding a contradiction. We conclude that there is a unitary $U$ and a projection $P\in \mathcal{A}$ so that 
\[E(PUSU^{*}P) = AP \quad \text{ and } \quad (I-P)USU^{*}(I-P) \approx A(I-P).\]

\end{proof}

We now prove the first of the two generalizations of the Schur-Horn theorem to type $II_1$ factors. We repeat the statement of the theorem for the convenience of the reader. 

\begin{theorem}\label{conj1}[The Schur-Horn theorem in type $II_1$ factors I]
Let $\mathcal{M}$ be a type $II_1$ factor and let $A, S \in \mathcal{M}$ be positive operators with $A \prec S$. Then, there is some masa $\mathcal{A}$ in $\mathcal{M}$ such that \[E_{\mathcal{A}}(S) \approx A.\]
\end{theorem}

\begin{proof}
Choose a masa $\mathcal{A}_1$ such that $A$ belongs to $\mathcal{A}_1$. Theorem (\ref{SHM1}) yields that there is a unitary $U$ in $M$ and a projection $P$ in $\mathcal{A}_1$ such that 
\[E_{\mathcal{A}_1}(PUSU^{*}P) = AP \quad \text{ and } \quad (I-P)USU^{*}(I-P) \approx A(I-P).\] Choose a masa $\tilde{\mathcal{A}}$ in $(I-P)\mathcal{M}(I-P)$ that contains $(I-P)USU^{*}(I-P)$. Then, we have that 
\begin{eqnarray}\label{FC}
E_{\mathcal{A}_1 P \oplus \tilde{\mathcal{A}}}(USU^{*}) = AP \oplus (I-P)USU^{*}(I-P) 
\end{eqnarray}

Note that we have that $A \approx AP \oplus (I-P)USU^{*}(I-P)$ . Let $\mathcal{A}$ be the masa $U^{*}(\mathcal{A}_1 P \oplus \tilde{\mathcal{A}})U$. We then get by applying the automorphism $\operatorname{Ad}(U^{*})$ to the equation (\ref{FC}) that 
\[E_{\mathcal{A}}(S) = U^{*}(AP \oplus (I-P)USU^{*}(I-P))U\]
Since $U^{*}(AP \oplus (I-P)USU^{*}(I-P))U \approx AP \oplus (I-P)USU^{*}(I-P) \approx A$, we are done. 
\end{proof}

\begin{remark}
The above theorem, as remarked in the introduction, is a natural generalization of the Schur-Horn theorem to type $II_1$ factors. This generalization does not directly imply the alternative conjecture of Arveson and Kadison from \cite{ArvKad}. We prove the conjecture in full in the next section.  
\end{remark}

\section{Proof of the Arveson-Kadison conjecture}
We now turn to the second natural generalization of the Schur-Horn theorem. The theorem of the last section characterizes the spectral distributions of operators that arise as the ``diagonal'' of a given positive operator $S$. On the other hand, the conjecture of Arveson and Kadison complements the abovementioned theorem by characterizing the spectral distributions of operators which have a prescribed diagonal $A$. 

The calculations in this section are straightforward but technical. Perhaps a few words about the idea of the proof might be helpful. Let $\mathcal{A}$ be a masa in type $II_1$ factor $\mathcal{M}$ and let $A \in \mathcal{A}$ and $S \in \mathcal{M}$ be positive elements so that $A \prec S$. Theorem (\ref{SHM1}) in the last section says that there is a unitary $U$ and a projection $P$ so that if we write write out $A$ and $USU^{*}$ in block matrix form with diagonal $P \oplus (I-P)$, 
 \[USU^{*}=\left( \begin{array}{cc}
S_1 & \ast\\
\ast & S_2 \end{array} \right), \qquad  A= \left( \begin{array}{cc}
A_1 & \ast\\
\ast & A_2 \end{array} \right), \]
then $E(S_1) = A_1$ and $S_2\cong A_2$ inside  $P\mathcal{M}P$ and $(I-P)\mathcal{M}(I-P)$ respectively. Even though $A_2$ and $S_2$ are approximately unitarily equivalent inside $(I-P)\mathcal{M}(I-P)$, we cannot expect to use these to implement an approximate unitary equivalence between $USU^{*}$ and an operator  of the form $\left( \begin{array}{cc}
S_1 & \ast\\
\ast & A_2 \end{array} \right)$. Nevertheless, there is a workaround; Let us look closely at what $A \prec S$ means in terms of the spectral scales $f_A$ and $f_S$. Roughly speaking, $f_S(x)$ is larger than $f_A(x)$ for $x$ close to $0$ and smaller for  $x$ close to $1$. Rather than work with $A$ and $S$ directly, we will work with $PAP$ and $QSQ$ where $P$ and $Q$ are carefully chosen spectral projections supported away from the extreme points of the spectra. We will apply theorem (\ref{SHM1}) to $PAP$ and $QSQ$ and might end up with pieces that are equimeasurable as above. We will then use the 'reserved' head and tail of the spectra that we have hitherto left untouched to massage the equimeasurable parts carefully, in order to achieve the desired diagonal.

The main result in this section is the proof of theorem (\ref{conj2}). We first prove a couple of lemmas. We will use the notation $f_S > f_A$ as shorthand for $f_S(x) > f_A(x)$ for all $x \in [0,1]$. 

\begin{lemma}\label{SHT2lem}
Let $\mathcal{A}$ be a masa in a type $II_1$ factor $\mathcal{M}$ and let $A \in \mathcal{A}$ and $S \in \mathcal{M}$ be two positive operators commuting with a projection $P$ of trace $\frac{1}{2}$ in $\mathcal{A}$, written with respect to the decomposition $I = P \oplus I-P$, 
\[S = \left( \begin{array}{cc}
S_{1} & 0\\
0 & S_2 \end{array} \right),   \quad A = \left( \begin{array}{cc}
A_1 & 0\\
0 & A_2 \end{array} \right)\]
where $f_{S_1}^{P\mathcal{M}P} - f_{A_1}^{P\mathcal{M}P} > 0$ and $\sigma^{P\mathcal{M}P}(A_1) \geq \sigma^{P\mathcal{M}P}(A_2)$ and further, $A_2 \approx S_2$. Then, for any $\delta > 0$, there is a unitary $U$ and  projections $R_1 \leq P$  and $R_2 \leq I-P$, both in $\mathcal{A}$, with $\tau(R_i) > \frac{1}{2} - \delta$ for $i = 1,2$ such that $E_{\mathcal{A}R_1}(R_1USU^{*}R_1) = A_1 R_1$ and $f_{R_2USU^{*}R_2}^{R_2\mathcal{M}R_2} - f_{A_2 R_2}^{R_2\mathcal{M}R_2} > 0$. 
\end{lemma}

\begin{proof}
Let $\delta  > 0$ be fixed. It is easy to see(using the fact that the spectral scales $f_A$ and $f_S$ are right continuous and non increasing) that we may find a natural number $k$, a number $\epsilon > 0$  and disjoint intervals $(a_1,a_1+\epsilon), \cdots, (a_{k+1},a_{k+1}+\epsilon)$ in $[0,1]$ with $a_1 < a_2 \cdots < a_k < a_{k+1}$ such that 
\begin{enumerate}
 \item $k \, \epsilon > 1 - 2\delta$. And, 
 \item $f_{S_1}^{P\mathcal{M}P}(x) - f_{A_1}^{P\mathcal{M}P}(y) > 0$  for $x,y \in [a_{i},a_{i}+\epsilon]$ for $i = 1, \cdots, k+1$.
\end{enumerate}

Define the projections
\[P_i:=\mu_{A_1}^{P\mathcal{M}P}((a_{i},a_{i}+\epsilon)), \quad Q_i:=\mu_{A_2}^{(I-P)\mathcal{M}(I-P)}((a_{i},a_{i}+\epsilon)), \quad i = 1, \cdots, k+1.\] Pick a unitary $U_1$ in $P\mathcal{M}P$ that  conjugates $\mu_{S_1}^{P\mathcal{M}P}((a_{i},a_{i}+\epsilon))$ onto $P_i$ and a unitary $U_2$ in $(I-P)\mathcal{M}(I-P)$ that conjugates $\mu_{S_2}^{(I-P)\mathcal{M}(I-P)}((a_{i},a_{i}+\epsilon))$ onto $Q_i$ for $i = 1, \cdots, k+1$. Let $U: = U_1 \oplus U_2$ and let $T:= USU^{*}$. Then, $T$ commutes with the projections $P_i$ and $Q_i$.

For $i = 1, \cdots k+1$, let $X_i := T(P_i \oplus Q_i)$ and $Y_i := A(P_i \oplus Q_i)$. Inside $(P_i \oplus Q_i) \mathcal{M} (P_i \oplus Q_i)$, we may write
\[X_i=\left( \begin{array}{cc}
X_i^1 & 0\\
0 & X_i^2 \end{array} \right), \qquad  Y_i= \left( \begin{array}{cc}
Y_i^1 & 0\\
0 & Y_i^2 \end{array} \right), \]
where 
\[\sigma_{P_i\mathcal{M}P_i}(X_i^1) = \sigma_{P_i\mathcal{M}P_i}(P_iTP_i) \geq f_{S_1}(a_i+\epsilon) > f_{A_1}(a_i) \geq \sigma_{P_i\mathcal{M}P_i}(P_iAP_i) = \sigma_{P_i\mathcal{M}P_i}(Y_i^1)\]
and since, $\sigma_{P\mathcal{M}P}(A_1) \geq \sigma_{P\mathcal{M}P}(A_2) = \sigma_{P\mathcal{M}P}(S_2)$, we have that
\[\sigma_{P_i\mathcal{M}P_i}(Y_i^1) \subset \sigma_{P\mathcal{M}P}(A_1) \geq \sigma_{(I-P)\mathcal{M}(I-P)}(S_2) \supset \sigma_{(I-P)\mathcal{M}(I-P)}(X_i^2)\]
We see that for each $i = 1, \cdots, k+1$, the pair of operators $Y_i=A (P_i \oplus Q_i)$ and $X_i=T (P_i \oplus Q_i)$ inside $(P_i \oplus Q_i) \mathcal{M} (P_i \oplus Q_i)$ satisfy the hypothesis of the remark following lemma(\ref{2matlem}) and we may thus find unitaries $V_i$ in $(P_i \oplus Q_i) \mathcal{M} (P_i \oplus Q_i)$ such that 
\begin{eqnarray}\label{fac1}
E_{\mathcal{A}P_i} (P_i V_i X_i V_i^{*} P_i) = Y_i P_i,  \quad \sigma (Q_i V_i X_i V_i^{*} Q_i) \subset \operatorname{int}{\operatorname{conv}(\sigma(X_i)) }
\end{eqnarray}

Let $W$ be a unitary in $(I-P)\mathcal{M}(I-P)$ that conjugates $Q_{i}$ onto $Q_{i+1}$ for $i = 2, \cdots k+1$, i.e. $WQ_{i}W^{*}=Q_{i+1}$. The second fact above gives us that 

\begin{eqnarray}\label{fac2}
\sigma(Q_{i+1} W V_i X_i V_i^{*} W^{*} Q_{i+1})=\sigma( Q_iV_i X_i V_i^{*} Q_i) > \sigma(Q_{i+1} A Q_{i+1}) 
\end{eqnarray}

Let $U$ be the unitary $(P \oplus W ) (I - \sum_i^{k+1} (P_i \oplus Q_i) + \sum_i^{k+1} V_i)$. Also, let $R_1 = \sum_1^{k+1} P_i$ and $R_2 = \sum_2^{k+1} Q_i$. The two facts, (\ref{fac1}) and (\ref{fac2}) give us that
\[E_{\mathcal{A}R_1}(R_1USU^{*}R_1) = A_1 R_1, \quad f_{R_2USU^{*}R_2}^{R_2\mathcal{M}R_2} - f_{A_2 R_2}^{R_2\mathcal{M}R_2} > 0\]
Finally, we have that $\tau(R_2) = k \epsilon$ and $\tau(R_1) = (k+1)\epsilon$ and both are greater than $1 - 2\delta$. We are done. 

\end{proof}

\begin{lemma}\label{Eqmlem}
Let $\mathcal{A}$ be a masa in a type $II_1$ factor $\mathcal{M}$ and let $A \in \mathcal{A}$ and $S \in \mathcal{M}$ be two positive operators commuting with a projection $P$ in $\mathcal{A}$, written with respect to the decomposition $I = P \oplus I-P$, 
\[S = \left( \begin{array}{cc}
S_{1} & 0\\
0 & S_2 \end{array} \right),  \quad A = \left( \begin{array}{cc}
A_1 & 0\\
0 & A_2 \end{array} \right)\] 
Assume $f_{S_1}^{P\mathcal{M}P} - f_{A_1}^{P\mathcal{M}P} > 0$ and $\sigma(A_1) \geq \sigma(A_2)$ and $A_2 \approx S_2$. Then, there is a projection $Q$ in $\mathcal{A}$ with $\tau(Q) \geq 1 - 2\tau(P)$ and a unitary $U$ such that 
\[E_{\mathcal{A}}(QUSU^{*}Q) = AQ \]
\end{lemma}

\begin{proof}
 We may assume that $\tau(P) \leq \frac{1}{2}$ for otherwise there is nothing to prove. Let $k$ be the natural number such that $(k+1) \tau(P) \leq 1 < (k+2)\tau(P)$; Note that $k > 1$. Now, choose a positive $\delta$ such that $ \dfrac{k(k+1)\delta}{2} < (k+2)\tau(P) - 1$. Next, define the sequence of numbers $\{a_1, \cdots, a_k\}$ using the following prescription: $a_1$ is such that 
 \[\tau(I-P)\tau_{(I-P)\mathcal{M}(I-P)}(\mu_{A_2}^{(I-P)\mathcal{M}(I-P)}((a_1,1))) = \tau(P)\]
  and for $i = 2, \cdots, k$, the number $a_i$ is such that 
  \[\tau(I-P)\tau_{(I-P)\mathcal{M}(I-P)}(\mu_{A_2}^{(I-P)\mathcal{M}(I-P)}((a_i,a_{i-1}))) = \tau(P)-(i-1)\delta.\] 

Now, define the sequence of projections $P_1, \cdots, P_k$ by
\begin{eqnarray}\label{pcal1}
P_1 := \mu_{A_2}^{(I-P)\mathcal{M}(I-P)}((a_1,1))
\end{eqnarray}
 and then for $i = 2, \cdots, k$, define 
\begin{eqnarray}\label{pcal2}
P_i := \mu_{A_2}^{(I-P)\mathcal{M}(I-P)}((a_i,a_{i-1})).
\end{eqnarray}
 We interpret these projections as lying in $\mathcal{M}$. Note that by (\ref{pcal1}) and (\ref{pcal2}), we have,
\[\tau(P_i) = \tau(P)-(i-1)\delta, \quad i = 1, \cdots, k\]
Pick a unitary $V$ such that 
\[V\mu_{S_2}^{(I-P)\mathcal{M}(I-P)}(a_1,1)V^{*} = \mu_{A_2}^{(I-P)\mathcal{M}(I-P)}(a_1,1)\] as well as 
 \[V\mu_{S_2}^{(I-P)\mathcal{M}(I-P)}((a_i,a_{i-1}))V^{*} = \mu_{A_2}^{(I-P)\mathcal{M}(I-P)}((a_i,a_{i-1})), \quad i = 1,\cdots k\]
  and which commutes with $P$, i.e. $V = P \oplus (I-P)V(I-P)$. We then have,
\[ A =  \left( \begin{array}{ccccc}
A_{1} & 0 & 0 & 0 & 0\\
0 & A_{21} & 0 & 0 & 0\\
0 & 0 & \ddots & 0 & 0\\
0 & 0 & 0 & A_{2k} & 0\\
0 & 0 & 0 & 0 & A_{2\, (k+1)} \end{array} \right), \quad  VSV^{*} =  \left( \begin{array}{ccccc}
S_{1} & 0 & 0 & 0 & 0\\
0 & S_{21} & 0 & 0 & 0\\
0 & 0 & \ddots & 0 & 0\\
0 & 0 & 0 & S_{2k} & 0\\
0 & 0 & 0 & 0 & S_{2\,( k+1)} \end{array} \right)\]
where $S_{2i} \approx A_{2i}$ for $i = 1, \cdots, k+1$ and $\sigma(S_{21}) \geq \cdots \geq \sigma(S_{2k}) \geq \sigma(S_{2\, k+1})$. Also, $f^{P\mathcal{M}P}_{S_{1}} - f^{P\mathcal{M}P}_{A_{1}} > 0$.

Applying lemma(\ref{SHT2lem}) to $A_{1} \oplus A_{21}$ and $S_{1} \oplus  S_{21}$ inside $(P\oplus P_1)\mathcal{M}(P\oplus P_1)$, we conclude that we may find a unitary $U_1$ commuting with $I-P-P_1$ and projections $Q_1$ and $R_1$ of trace $\tau(P)-\delta$ in $\mathcal{A}$ with $Q_1 \leq P$ and $R_1 \leq P_1$ such that letting $T_1 = U_1 S U_1^{*}$, we have that  $E_{\mathcal{A}}(Q_1 T_1 Q_1) = AQ_1$ and  further, $f_{R_1 T_1 R_1}^{R_1\mathcal{M}R_1} - f_{A_{21} R_1}^{R_1\mathcal{M}R_1} > 0$.

Inductively, for $i = 1, \cdots, k-1$, do the following: Note that $f_{R_i T_i R_i}^{R_i\mathcal{M}R_i} - f_{A_{2i} R_i}^{R_i\mathcal{M}R_i} > 0$ and apply lemma(\ref{SHT2lem}) together with the remark (\ref{greater}) following it to $A_{2i}R_i \oplus A_{2\, i+1}$ and $T_i R_i\oplus  S_{2\, i+1}$ inside $(R_i \oplus P_{i+1})\mathcal{M}(R_i \oplus P_{i+1})$. The lemma yields a unitary $U_{i+1}$ commuting with $I-R_i-P_{i+1}$ and projections $Q_{i+1}$ and $R_{i+1}$ of trace $\tau(P)-i\delta$ in $\mathcal{A}$ with $Q_{i+1} \leq P_{i}$ and $R_{i+1} \leq P_{i+1}$ such that letting $T_{i+1} = U_{i+1} T_i U_{i+1}^{*}$, we have 
\begin{enumerate}
\item  $E_{\mathcal{A}Q_{i+1}}(Q_{i+1} T_{i+1} Q_{i+1}) = AQ_{i+1}$ 
\item Since $U_{i+1} = (R_i+P_{i+1})U_i(R_i + P_{i+1}) + (I-R_i-P_{i+1})$, we have $(T_{i+1}-T_{i})(I-R_i-P_{i+1}) = 0$ and hence, $E_{\mathcal{A}Q_{j}}(Q_{j} T_{i+1} Q_{j}) = AQ_{j}$ for $j \leq i$ as well.

\item $f_{R_{i+1} T_{i+1} R_{i+1}}^{R_{i+1}\mathcal{M}R_{i+1}} - f_{A_{2\,i+1} R_{i+1}}^{R_{i+1}\mathcal{M}R_{i+1}} > 0$.
\end{enumerate}

Putting it all together, letting $U = U_1U_2 \cdots U_{k}$, we have that 
\[E_{\mathcal{A}Q}(QUSU^{*}Q) = AQ\]
where $Q = Q_1 \oplus Q_2 \cdots \oplus Q_k$. We have that $\tau(Q_i) = \tau(P) - i\delta$ and thus, 
\begin{eqnarray*}
 \tau(Q) &=& \sum_{i=1}^k \tau(P) - i\delta\\
&=& k\tau(P) - \dfrac{k(k+1)}{2} \delta\\
&>& 1 - 2\tau(P)
\end{eqnarray*}

\end{proof}

\begin{lemma}\label{Eqmlem2}
Let $\mathcal{A}$ be a masa in a type $II_1$ factor $\mathcal{M}$ and let $A \in \mathcal{A}$ and $S \in \mathcal{M}$ be two positive operators commuting with a projection $P$ in $\mathcal{A}$, written with respect to the decomposition $I = P \oplus I-P$, 
\[S = \left( \begin{array}{cc}
S_{1} & 0\\
0 & S_2 \end{array} \right),  \quad A = \left( \begin{array}{cc}
A_1 & 0\\
0 & A_2 \end{array} \right)\] 
Assume $f_{S_1}^{P\mathcal{M}P} - f_{A_1}^{P\mathcal{M}P} > 0$ and $\sigma(A_1) \geq \sigma(A_2)$ and $A_2 \prec S_2$ inside $(I-P)\mathcal{M}(I-P)$. Then, there is a projection $Q$ in $\mathcal{A}$ with $\tau(Q) > 1 - 2\tau(P)$ and a unitary $U$ such that 
\[E_{\mathcal{A}}(QUSU^{*}Q) = AQ \]
\end{lemma}

\begin{proof}
Applying theorem (\ref{conj1}) to $A_2$ and $S_2$ inside $(I-P)\mathcal{M}(I-P)$, we see that there is a unitary of the form $V = P \oplus \tilde{V}$ and a projection $P_1$ in $\mathcal{A}$ smaller than $(I-P)$ so that with respect to $I = P \oplus P_1 \oplus (I-P-P_1)$, we have
\[VSV^{*} = \left( \begin{array}{ccc}
S_{1} & 0 & 0\\
0 & S_{3} & \ast\\
0 & \ast & S_{d} \end{array} \right),  \quad A = \left( \begin{array}{ccc}
A_1 & 0 & 0\\
0 & A_{3} & 0\\
0 & 0 & A_d \end{array} \right)\] 
where 
\[S_{3} \approx A_{3}, \quad E_{\mathcal{A}P_1}(S_{d}) = A_{d}, \]

Compressing to $(P\oplus P_{1})\mathcal{M}(P\oplus P_{1})$, we apply lemma (\ref{Eqmlem}) to $(P \oplus P_{1})\mathcal{M}(P \oplus P_{1})$ and $A(P \oplus P_{1})$ to conclude that there is a projection $R$ in $\mathcal{A}$, smaller than $P \oplus P_1$ of trace greater than $\tau(P_1)-\tau(P)$ and a unitary $W$ such that $E_{\mathcal{A}}(RWVSV^{*}W^{*}R) = AR$.  Let $U = VW$; together with the fact that the operator $S_{d}$ has the ''right diagonal'', we conclude that 
\[E_{\mathcal{A}}(QWSW^{*}Q) = AQ\]
where $Q= R \oplus  I - P - P_1$. We see that 
\[\tau(Q) = \tau(R) + \tau(I-P-P_1) > \tau(P_1) - \tau(P) + \tau(I-P-P_1) = 1 - 2\tau(P)\]

\end{proof}

We now turn to the main theorem of the paper, the proof of the conjecture (\ref{conj2}) of Arveson and Kadison in \cite{ArvKad}. We start off with some preliminary remarks. Let $f_A$ and $f_S$ be the spectral scales of $A$ and $S$ respectively. Define \[\mathcal{E}:= \{x \in (0,1) : f_A(x) = f_S(x)\}\]
Choose a unitary that conjugates $ \mu_S(\mathcal{E})$ onto  $\mu_A(\mathcal{E})$. With respect to the decomposition $I = \mu_A(\mathcal{E}) \oplus \mu_A(\mathcal{E}^c) = (I-P) \oplus P$, we may write 
\[A = \left( \begin{array}{cc}
A_1 & 0\\
0 & A_2 \end{array} \right)  \quad USU^{*} =   \left( \begin{array}{cc}
S_1 & 0\\
0 & S_2 \end{array} \right) \]
Then, $A_1 \approx S_1$ and $A_2 \prec S_2$ inside $P \mathcal{M} P$. It is now easy to see that if we can prove the theorem for $A_2$ and $S_2$ inside $P \mathcal{M}P$, the result for $A$ and $S$ inside $\mathcal{M}$ would follow. We may therefore assume that $f_A \neq f_S$ almost everywhere on $[0,1]$. 

Let $F_A$ and $F_S$ be the Ky Fan norm functions. The relation $A \prec S$ gives us that $F_A \leq F_S$ on $[0,1]$. 
Define  \[\mathcal{F}:=\{x \in (0,1) : F_A(x) = F_S(x)\}\]
Since we assume that $f_A \neq f_S$ almost everywhere on $[0,1]$, $\mathcal{F}$ cannot contain any intervals. We may write $F^{c}$ as a union of disjoint intervals $\{I_{\alpha}\}$; Pick a unitary $U$ that conjugates $ \mu_S(I_{\alpha})$ onto  $\mu_A(I_{\alpha})$ for every $\alpha$. Then,
\[A = \sum_{\alpha} A \mu_{A}(I_{\alpha}) \quad \text{and} \quad USU^{*} = \sum_{\alpha} USU^{*} \mu_{A}(I_{\alpha})\] 
where $A \mu_{A}(I_{\alpha}) \prec USU^{*} \mu_{A}(I_{\alpha})$ and further the corresponding Ky Fan norm functions are strictly positive on $(0,1)$ for every $\alpha$. It is routine to see that if we can solve the problem for every $\alpha$, the general theorem follows. Therefore, we may assume, additionally to $f_A \neq f_S$ a.e. on $[0,1]$, that $F_A < F_S$ on $(0,1)$. 

\begin{remark}\label{HonMaj}
We use the following notation:
\begin{eqnarray}
 A \lnsim_{w} S \quad \text{ if } \quad A \prec_{w} S,\quad  F_A < F_S \text{ on } (0,1)
\end{eqnarray}
If we have that $A \lnsim_{w} S$ and also $\tau(A) = \tau(S)$, we say that $A \lnsim S$. 
\end{remark}

\begin{proposition}\label{SH1Part}
Let $A \in \mathcal{A}$ and $S \in \mathcal{M}$ be positive operators with $A \lnsim S$. Then, there is a projection $P$ in $\mathcal{A}$ with $\tau(P) \geq \frac{1}{2}$ and a unitary $U$ in $\mathcal{M}$ such that $E(PUSU^{*}P) = AP$ and $A(I-P) \lnsim (I-P)USU^{*}(I-P)$.  
\end{proposition}

\begin{proof}

 By assumption, we have that  $F_S > F_A$ on $(0,1)$. Let $a$ and $b$ be numbers with $a < \frac{1}{8}$ and $b > \frac{7}{8}$, so that  
\begin{eqnarray}\label{calc1}
F_S(x) - F_A(x) > F_S(a) - F_A(a) = F_S(b) - F_A(b), \quad a < x < b.
\end{eqnarray}
This can be done as follows; Recall that the functions $F_S$ and $F_A$ are continuous on $[0,1]$.  Let $\alpha := \operatorname{min}(\{F_S(x)-F_A(x) : x \in [\frac{1}{8},\frac{7}{8}]\})$. Now, let
\[a =\operatorname{sup}(\{x : F_S(x)-F_A(x) = \frac{\alpha}{2}, \, x \leq \frac{1}{8}\}), \quad b = \operatorname{sup}(\{x : F_S(x)-F_A(x) = \frac{\alpha}{2}, \, x \geq \frac{7}{8}\}).\]
The function $F_S-F_A$ is greater than $\frac{\alpha}{2}$ on $[a,\frac{1}{8}]$ and $[\frac{7}{8},b]$ by the choice of $a$ and $b$ and is at least $\alpha$ on $[\frac{1}{8},\frac{7}{8}]$. Assertion (\ref{calc1}) follows. Note further that $a < \frac{1}{8}$ and $b > \frac{7}{8}$. Since $f_A$ and $f_S$ are right continuous

Now, choose a unitary $V_1$ so that
\[V_1 \mu_S((0,a))V_1^{*}= \mu_A((0,a)), \quad  V_1\mu_S((a,b))V_1^{*} = \mu_A((a,b)), \quad  V_1\mu_S((a,b))V_1^{*} = \mu_A((b,1))\]

 With respect to the decomposition $I = Q_1 \oplus Q_2 \oplus Q_3 = \mu_A((0,a)) \oplus \mu_A((a,b)) \oplus \mu_A((b,1))$, we write
\[ A =   \left( \begin{array}{ccc}
A_1 & 0 & 0\\
0 & A_2 & 0\\
0 & 0 & A_3\end{array} \right) \quad V_1SV_1^{*} =   \left( \begin{array}{ccc}
S_1 & 0 & 0\\
0 & S_2 & 0\\
0 & 0 & S_3\end{array} \right)\]

Note that 
\begin{eqnarray}\label{ab}
\tau(Q_1)=a, \quad \tau(Q_2)  = b - a > \dfrac{7}{8} - \dfrac{1}{8} = \dfrac{3}{4}. 
\end{eqnarray}
Now,
\[\tau(Q_2AQ_2) = \tau(S\mu_{A}(a,b)) = \int_a^b f_A(x)dm(x) = F_A(b)-F_A(a) = F_S(b)-F_S(a)=\tau(Q_2SQ_2)\]
It is easy to see that $f_{Q_2AQ_2}^{Q_2\mathcal{M}Q_2}(x) = f_{S}(\tau(Q_1)+\tau(Q_2)x)$ and thus, for $x \in (0,1)$, 
\begin{eqnarray*}
F_{Q_2SQ_2}^{Q_2\mathcal{M}Q_2}(x)-F_{Q_2AQ_2}^{Q_2\mathcal{M}Q_2}(x) &=&\dfrac{\int_{0}^{x} [f_{S}(\tau(Q_1)+\tau(Q_2)y) -f_{A}(\tau(Q_1)+\tau(Q_2)y)]dm(y)}{\tau(Q_2)}\\
&=&\int_{\tau(Q_1)}^{\tau(Q_1)+x\tau(Q_2)} [f_{S}(y) -f_{A}(y)]dm(y)\\
&=&[F_S(\tau(Q_1)+x\tau(Q_2))-F_A(\tau(Q_1)+x\tau(Q_2))]-[F_S(a)-F_A(a)]\\
&>& 0
\end{eqnarray*}
The last inequality is because for every $x \in (0,1)$, we have $\tau(Q_1)+x\tau(Q_2) = a+(b-a)x = a(1-x)+bx \subset (a,b)$ and thence because of (\ref{calc1}).

We therefore have that 
\begin{eqnarray}\label{precst}
A_2 \prec S_2 \text{ inside } \, Q_2 \mathcal{M} Q_2.
\end{eqnarray}
 Similarly, we can prove that
\begin{eqnarray}\label{precst2}
A_1 \oplus A_3 \prec S_1 \oplus S_3 \text{  inside }\, (Q_1 + Q_3)\mathcal{M}(Q_1 + Q_3).
\end{eqnarray}

 Note further that 
\begin{eqnarray}\label{precst3}
\sigma(A_1) \geq \sigma(A_2) \geq \sigma(A_3), \quad \sigma(S_1) \geq \sigma(S_2) \geq \sigma(S_3).
\end{eqnarray}

Noting that $F_S(a)-F_A(a_1)>0$, we may choose a number $c$ in $(0,a)$ such that $f_{S}(c) > f_{A}(c)$. Recall that the spectral scales are right continuous. Thus, we may find an interval $I=[c,d]$ such that $f_{S} > f_{A}$ on $I$.  Now, choose a $\delta > 0$ satisfying $c+\delta < a$ and 
\begin{eqnarray}\label{delcho}
 \delta < \dfrac{\tau(Q_{2})}{3}, \quad f_{S} > f_{A}\,\,\, \text{ on } [c,c+\delta],
\end{eqnarray}
as well as(by passing to a smaller $\delta$ if needed),
\begin{eqnarray}\label{delcho2}
\operatorname{min}_{x \in [c,b]} \{F_{S}(x) - F_{A}(x)\} >  3\delta
\end{eqnarray}

Let us define 
\[Q_{11}:=\mu_{A}((0,c)), \quad Q_{12}:= \mu_{A}((c,c+\delta)),\quad  Q_{13}:=\mu_{A}((c+\delta,a))\]

Now choose a unitary  $W$ that is the identity on $I-Q_1$ such that
\[W\{\mu_{S}((0,c)), \mu_{S}((c,c+\delta)), \mu_{S}((c+\delta,a))\}W^{*}=\{\mu_A((0,c)), \mu_A((c,c+\delta)), \mu_A((c+\delta,a))\}\]
 Now, with respect to $I = Q_{11}\oplus Q_{12}\oplus Q_{13} \oplus Q_{2}  \oplus Q_{3}$, we may write
\[ A =  \left( \begin{array}{ccccc}
A_{11} & 0 & 0 & 0 & 0\\
0 & A_{12} & 0 & 0 & 0\\
0 & 0 & A_{13} & 0 & 0\\
0 & 0 & 0 & A_{2} & 0\\
0 & 0 & 0 & 0 & A_3\end{array} \right), \quad  WSW^{*} =  \left( \begin{array}{ccccc}
S_{11} & 0 & 0 & 0 & 0\\
0 & S_{12} & 0 & 0 & 0\\
0 & 0 & S_{13} & 0 & 0\\
0 & 0 & 0 & S_{2} & 0\\
0 & 0 & 0 & 0 & S_3\end{array} \right)\]
Compress to $(Q_{12} +Q_{2}) \mathcal{M} (Q_{12} + Q_{2})$; let $B = A(Q_{12} + Q_{2}) $, $R = WVSV^{*}W^{*}(Q_{12} + Q_{2})$,
\[A(Q_{12} + Q_{2}) \sim \left( \begin{array}{cc}
A_{12} & 0\\
0 & A_{2} \end{array} \right), \quad R \sim \left( \begin{array}{cc}
S_{12} & 0\\
0 & S_{2} \end{array} \right)\]
Note that we have the following, 
\begin{eqnarray*}
A_{2} \prec S_{2},  \quad f_{S_{12}}^{Q_{12}\mathcal{M}Q_{12}} > f_{A_{12}}^{Q_{12}\mathcal{M}Q_{12}}, \quad \sigma_{Q_{12}\mathcal{M}Q_{12}}(A_{12}) \geq \sigma_{Q_{2}\mathcal{M}Q_{2}}(A_{2})\end{eqnarray*}
The assertions follow from (\ref{precst}), (\ref{precst3}) and (\ref{delcho}) respectively.

Applying lemma(\ref{Eqmlem2}) to $B$ and $R$, we get a projection $P$ in $\mathcal{A}$ with $P \leq Q_{12} \oplus Q_{2}$ and a unitary $W_1$ that is the identity on $I-Q_{12}-Q_{2}$ such that letting $U=W_1WV$, we have,
\[E_{\mathcal{A}}(PUSU^{*}P) = AP \]
and also,
\[\tau(P) > \tau(Q_{12}+Q_{2}) - 2\tau(Q_{12}) = \tau(Q_{2})-\tau(Q_{12})\]
By our choice of $Q_{12}$, $\tau(Q_{12}) = \delta < \dfrac{\tau(Q_2)}{3}$ and hence, 
\[\tau(P) > \dfrac{2}{3}\tau(Q_2) = \dfrac{2}{3}(b-a) > \dfrac{2}{3}(\dfrac{3}{4}) = \dfrac{1}{2}\]
The penultimate inequality follows from (\ref{ab}).

Let  $\tilde{P} = Q_{12}+Q_{2}-P$ and note that $\tau(\tilde{P}) < 2\tau(Q_{12}) = 2\delta$. We have with respect to $I = Q_{11} \oplus P \oplus Q_{13} \oplus \tilde{P}  \oplus Q_3$, 
\[  A =  \left( \begin{array}{ccccc}
A_{11} & 0 & 0 & 0 & 0\\
0 & B & 0 & 0 & 0\\
0 & 0 & A_{13} & 0 & 0\\
0 & 0 & 0 & C & 0\\
0 & 0 & 0 & 0 & A_3\end{array} \right), \quad  USU^{*} =  \left( \begin{array}{ccccc}
S_{11} & 0 & 0 & 0 & 0\\
0 & X & 0 & \ast & 0\\
0 & 0 & S_{13} & 0 & 0\\
0 & \ast & 0& Y & 0\\
0 & 0 & 0 & 0 & S_3\end{array} \right)\]
and where $E_{\mathcal{A}P}(X) = B$. We also have that
\begin{eqnarray}\label{spec}
 \sigma_{\tilde{P}\mathcal{M}\tilde{P}}(Y)\subset \sigma_{P\mathcal{M}P}(S_{12})\cup \sigma_{Q_{13}\mathcal{M}Q_{13}}(S_{13}) \cup \sigma_{\tilde{P}\mathcal{M}\tilde{P}}(S_2)
\end{eqnarray}

The last step in the proof is to show that  
\[A(I-P) \sim \left( \begin{array}{cccc}
A_{11} &  0  & 0 & 0\\
0 & A_{13} & 0 & 0\\
 0 & 0 & C & 0\\
 0 & 0 & 0 & A_{3} \end{array} \right) \lnsim  \left( \begin{array}{cccc}
S_{11} & 0 & 0\\
0 &  S_{13}  & 0 & 0 \\
0 & 0 & Y & 0\\
0 & 0 & 0 & S_3\end{array} \right) \sim  (I-P)USU^{*}(I-P).\] 
The projection decomposition is with respect to $Q_1 \oplus Q_{13} \oplus \tilde{P} \oplus Q_3$. Now, let us write the projection $\tilde{P}+Q_{13}$ as $Q_4$ and the operator $Q_4A$ as $A_4$ and $Q_4USU^{*}Q_4$ as $S_4$, that is, 
\[A_4 = \left( \begin{array}{cc}
A_{13} & 0\\
0 & C \end{array} \right), \quad S_4 = \left( \begin{array}{cc}
S_{13} & 0\\
0 & Y \end{array} \right), \quad \text{ inside }Q_4\mathcal{M}Q_4\]
Note that
\[\sigma_{Q_{11}\mathcal{M}Q_{11}}(A_{11}) \geq \sigma_{Q_{4}\mathcal{M}Q_{4}}(A_4) \geq \sigma_{Q_{3}\mathcal{M}Q_{3}}(A_3), \] as well as by (\ref{spec}),
\[\sigma_{Q_{11}\mathcal{M}Q_{11}}(S_{11}) \geq \sigma_{Q_{4}\mathcal{M}Q_{4}}(S_4) \geq \sigma_{Q_{3}\mathcal{M}Q_{3}}(S_3), \]

The condition (\ref{calc1}) implies that
\[\tau(Q_{13})L_{Q_{13}\mathcal{M}Q_{13}}(A_{13},S_{13}) + \tau((Q_{11}+Q_{12})(S-A)) = \operatorname{inf}_{x \in (c,a)} F_S(x)-F_A(x) > 3\delta\]
and hence,
\begin{eqnarray}\label{13cal}
 \tau(Q_{13})L_{Q_{13}\mathcal{M}Q_{13}}(A_{13},S_{13}) &>& 3\delta - \tau(Q_{11}(S-A))+\tau(Q_{12}(S-A)) \\
\nonumber &\geq& 2\delta - \tau(Q_{11}(S-A))
\end{eqnarray}

We have by statement (2) of (\ref{Llemma}),
\begin{eqnarray}\label{4com}
\tau(Q_{4})L_{Q_4\mathcal{M}Q_4}(A_4,S_4) &\geq& \tau(Q_{13})L_{Q_{13}\mathcal{M}Q_{13}}(A_{13},S_{13})+\tau(\tilde{P})L_{\tilde{P}\mathcal{M}\tilde{P}}(C,Y)\\
\nonumber &\geq& \tau(Q_{13})L_{Q_{13}\mathcal{M}Q_{13}}(A_{13},S_{13})-\tau(\tilde{P})\\
\nonumber &>& \tau(Q_{13})L_{Q_{13}\mathcal{M}Q_{13}}(A_{13},S_{13}) - 2\delta\\
\nonumber &>&  -\tau(Q_{11}(S-A))
\end{eqnarray}

We may write
\[A(I-P) \sim \left( \begin{array}{ccc}
A_{11} &  0  & 0 \\
0 & A_{4} & 0\\
 0 & 0 & A_{3} \end{array} \right), \quad   (I-P)USU^{*}(I-P) \sim \left( \begin{array}{ccc}
S_{11} & 0 & 0\\
0 &  S_4 & 0 \\
0 & 0 & S_3\end{array} \right) .\] 

The calculation (\ref{4com})  implies that
\[\tau(Q_{11}(S-A)) + \tau(Q_{4})L_{Q_4\mathcal{M}Q_4}(A_4,S_4) >0\]

By lemma(\ref{lem1}), we see that
\[A(I-P) \lnsim (I-P)USU^{*}(I-P)\]

\end{proof}

\begin{corollary}\label{cor2}
Let $A \in \mathcal{A}$ and $S \in \mathcal{M}$ be positive operators and suppose we have a projection $P$ in $\mathcal{A}$ and a unitary $U$ in $\mathcal{M}$ such that 
\[E(PUSU^{*}P) = AP, \quad A(I-P) \lnsim (I-P)USU^{*}(I-P).\]
Then, there is a projection $Q$ in $\mathcal{A}$ such that  $Q > P$ with $\tau(I-Q) \leq \dfrac{\tau(I-P)}{2}$ and a unitary $V$ in $\mathcal{M}$ such that 
\[P(V-U) = 0, \quad E(QVSV^{*}Q) = AQ, \quad A(I-Q) \lnsim (I-Q)VSV^{*}(I-Q).\]
 
\end{corollary}

 This is proved in the same way that corollary(\ref{cor1}) is deduced from proposition (\ref{prop1}) and we omit the proof, using proposition(\ref{SH1Part}) in place of proposition (\ref{prop1}).

Corollary (\ref{cor2})  will imply the main Schur-Horn theorem. The passage from a partial solution to the full solution of the problem can be done exactly as in the proof of theorem (\ref{conj1}). 

\begin{theorem}\label{conj2}[The Schur-Horn theorem in type $II_1$ factors II]
Let $\mathcal{A}$ be a masa in a type $II_1$ factor $\mathcal{M}$. If $A \in \mathcal{A}$ and $S \in \mathcal{M}$ are positive operators with $A \prec S$. Then, there is an element $T \in \mathcal{O}(S)$ such that \[E(T) = A\]
\end{theorem}

\begin{proof}
Assume first that $A \lnsim S$. Using proposition (\ref{SH1Part}) and corollary (\ref{cor2}), we may pick a sequence of projections $\{P_n\}$ in $\mathcal{A}$ and a sequence of unitaries $\{U_n\}$ of $\mathcal{M}$ such that  
\[ E(P_nU_sSU_n^{*}P_n) = AP_n, \quad A(I-P_n) \lnsim (I-P_n)U_n SU_n^{*}(I-P_n)\]
as well as
\[P_n U_n = P_n U_{n+1}, \quad n = 1, 2, \cdots \]
We may choose the $P_n$ so that  
\[\tau(P_1) \geq \dfrac{1}{2}, \quad \tau(I-P_{n+1}) \leq \dfrac{\tau(I-P_n)}{2}\]
 yielding that $\tau(P_n) \geq 1-\dfrac{1}{2^n}$. It is now routine to see that the unitaries $U_n$ converge in the strong operator topology to a unitary that we denote $U$ and that we have 
\[E(USU^{*}) = A.\]

For the general case, as in the discussion preceding proposition (\ref{SH1Part}), we can find a unitary $V$ and a projection $Q$ in $\mathcal{A}$ so that with respect to $I = I-Q \oplus Q$,
\[ A =   \left( \begin{array}{cc}
A_1 & 0\\
0 & A_2 \end{array} \right) \quad VSV^{*} =   \left( \begin{array}{cc}
S_1 & 0\\
0 & S_2 \end{array} \right) \]
where $A_1 \approx S_1$ and $A_2 \lnsim S_2$. 
Then, there is a unitary $U$ of the form $I \oplus U_1$  such that $E_{\mathcal{A}Q}(U_1 S U_1^{*}) = AQ$. It is routine to see that the operator $T$ defined by
\[ T :=  \left( \begin{array}{cc}
A_1 & 0\\
0 & U_1S_2U_1^{*} \end{array} \right) \]
is such that  $E(T)=A$ and that $T$ is in $\mathcal{O}(S)$.
\end{proof}

We record one consequence that emerged in the above proof separately.

\begin{theorem}\label{Uthm}
Let $\mathcal{A}$ be a masa in a type $II_1$ factor $\mathcal{M}$. If $A \in \mathcal{A}$ and $S \in \mathcal{M}$ are positive operators with $A \prec S$. Assume further that $ A \lnsim S$, that is, $F_A(x) < F_S(x)$ for all $x \in (0,1)$ Then, there is a unitary $U$ so that   
\[E(USU^{*}) = A\]

\end{theorem}

This last theorem has a nice consequence ; There is no need to take the norm closure of the unitary orbit to achieve a desired diagonal when the diagonal has finite spectrum. To prove this theorem, we need Choquet's notion of comparison of measures : Given two regular Borel measures $\mu$ and $\nu$ on $\mathbb{R}$, we say that $\mu \prec \nu$ if for every tuple of positive Borel measures $\mu_1, \cdots, \mu_m$ such that 
 $\sum_{i=1}^{m} \mu_i = \mu$, there are positive Borel measures $\nu_1,\cdots,\nu_m$ such that $\sum_{i=1}^{m} \nu_i = \nu$ and such that $\int_{\mathbb{R}}xd\mu_i = \int_{\mathbb{R}}xd\nu_i$ for $i = 1\ \cdots,m$.
 
  Let $A$ and $S$ be two positive operators  in a type $II_1$ factor $\mathcal{M}$, with spectral measures $\mu_A$ and $\mu_S$. We let $\tau(\mu_A)$ denote the scalar measure on $\mathbb{R}$ given by $X \rightarrow \tau(\mu_A(X))$ where $X$ is any Borel set and similarly for $S$. It is a basic fact that the following are equivalent for positive operators $A$ and $S$ in type $II_1$ factors, see \cite{HiCh},
\begin{enumerate}
\item $A \prec S$
\item $\tau(\mu_A) \prec \tau(\mu_S)$. 
 \end{enumerate}
Interpreting statement (2) above operator algebraically, we see that $A \prec S$ is equivalent to saying that for every partition into projections, $P_1 + \cdots + P_k = I$ commuting with $A$ we have a partition into projections $Q_1 + \cdots + Q_k = I$ commuting with $S$ so that
 \[\tau(P_m) = \tau(Q_m)\quad \text{and} \quad \tau(AP_m) = \tau(SQ_m)\quad \text{for} \quad 1 \leq m \leq k\]
\begin{corollary}
Let $\mathcal{M}$ be a type $II_1$ factor, $\mathcal{A}$ a masa in $\mathcal{M}$, $S$ a positive operator in $\mathcal{M}$ and $A$ a positive operator in $\mathcal{A}$ with finite spectrum so that $A \prec S$. Then, there is a unitary $U$ so that $E(USU^{*}) = A$. 
\end{corollary}
\begin{proof}
Let us first assume that $A$ is a scalar, that is, $A = \tau(S)I$. If $S$ was a scalar as well, there is nothing to prove. Let us therefore assume that $S \neq \tau(S)I$. Since $f_S$ is non-increasing, we see that for any $0 < x < 1$, $\dfrac{\int_0^x f_S(t)dm(t)}{x} \geq \int_0^1 f_S(t)dm(t) = \tau(S)$ with equality precisely when $S = \tau(S)I$. Hence, $F_A(x) = x\tau(S) < F_S(x)$ on $(0,1)$. Theorem(\ref{Uthm})gives us a unitary $U$ so that 
\[E(USU^{*}) = \tau(S)I = A\]

Now suppose $A$ has finite spectrum; there are projections $P_1, \cdots, P_n$ summing up to $I$ and numbers $a_1, \cdots, a_n$ so that $A = a_1 P_1 + \cdots + a_n P_n$. By Choquet's comparison of measures, there are projections $Q_1, \cdots, Q_n$ commuting with $S$ and summing up to $I$ so that $\tau(Q_m) = \tau(P_m)$ and $\tau(AP_m) = \tau(SQ_m)$ for $m = 1, \cdots, n$. Choose a unitary $U$ so that $UQ_mU^{*}=P_m$ for $m = 1, \cdots, n$. We see that for every $m$, $USU^{*}$ commutes with $P_m$ and that we have $\sigma_{P_m\mathcal{M}P_m}(AP_m) = \{a_m\}$. By the result for scalar diagonals, we have projection $V_m$ in $P_m \mathcal{M}P_m$ so that $E(V_m (USU^{*}P_m)V_m^{*}) = a_m P_m$. Let $W = (V_1 + \cdots + V_n)U$; It is easy to check that 
\[E(WSW^{*}) = A\]

\end{proof}

\section{The Schur-Horn theorem in type $II_{\infty}$ factors}

The Schur-Horn theorem in type $II_1$ factors allows us to quickly prove an analogous theorem for trace class operators in type $II_{\infty}$ factors. One thing to note is that not all masas in type $II_{\infty}$ factors admit normal conditional expectations. It is a result of Takesaki\cite{Tak} that if all masas in a von Neumann algebra admit normal conditional expectations, then the von Neumann algebra is finite. Masas in type $II_{\infty}$ factors that do admit normal conditional expectations are generated by their finite projections - We will refer to these as atomic masas in analogy to $\mathcal{B}(\mathcal{H})$.

In \cite{ArvKad}, Arveson and Kadison proved a Schur-Horn theorem for trace class operators in $\mathcal{B}(\mathcal{H})$; We prove an exact analogue of their result here. The proof follows from a routine reduction to the $II_1$ factor case, which we accomplish by

\begin{lemma}
Let $\mathcal{A}$ be a atomic masa in a type $II_{\infty}$ factor $\mathcal{M}$ and let $A \in \mathcal{A}$ and $S \in \mathcal{M}$ be positive trace class operators so that $A \prec S$. Then, there is a unitary and a finite projection $P$ in $\mathcal{A}$ so that $USU^{*}$ commutes with $P$ and  
\[AP \prec USU^{*}P  \quad \text \quad  A(I-P) \prec USU^{*}(I-P) \]
\end{lemma}
\begin{proof}
 The proof is identical to the first part of the proof of theorem(\ref{conj1}) and we omit it. 
\end{proof}

The lemma yields a straighforward corollary
\begin{corollary}\label{corti}
 Let $\mathcal{A}$ be an atomic masa in a type $II_{\infty}$ factor $\mathcal{M}$ and let $A \in \mathcal{A}$ and $S \in \mathcal{M}$ be positive trace class operators so that $A \prec S$. Then, there is a unitary $U$ and a countable set of orthogonal finite projections $\{P_{n}\}$ in $\mathcal{A}$ summing up to $I$ so that $USU^{*}$ commutes with each projection $P_{n}$ and 
\[AP_{n}  \prec USU^{*}P_{n} \quad \forall n\]
\end{corollary}
\begin{proof}
 This a routine induction argument and we omit it. 
\end{proof}

Recall that for trace class operators in type $II_{\infty}$ factors, we have defined $\mathcal{O}(\mathcal{S})$ as the closure of the unitary orbit in the trace norm, see (\ref{TCO}). It is wasy to see that for a positive operator inside a type $II_1$ factor, the closures in the operator norm and the trace norm coincide(with the set of operators equimeasurable to the given one). The Schur Horn theorem for trace class operators in type $II_{\infty}$ factors is as follows
\begin{theorem}\label{SHTIN}
  Let $\mathcal{A}$ be an atomic masa in a type $II_{\infty}$ factor $\mathcal{M}$ and let $A \in \mathcal{A}$ and $S \in \mathcal{M}$ be positive trace class operators so that $A \prec S$. Then, there is an operator $T \in \mathcal{O}(S)$ so that 
\[E(T) = S\]
where $E$ is the canonical $\tau$ preserving conditional expectation onto $\mathcal{A}$.
\end{theorem}
\begin{proof}
 Corollary(\ref{corti}) yields us a unitary $U$ and a countable set of orthogonal projections, $\{P_{n}\}$ so that $USU^{*} = \sum P_{n} USU^{*}$ and so that $AP_n \prec USU^{*}P_{n}$. Applying theorem(\ref{conj2}) to each of the $II_1$ factors $P_{n} \mathcal{M} P_{n}$ yields us a set of operators $T_{n} \in \mathcal{O}(P_{n} USU^{*}) \in P_{n} \mathcal{M}P_{n}$ such that $E(T_{n}) = A_{n}$. Now, let $T = \sum_{n} T_{n}$ and fix an $\epsilon > 0$.

  Since $T_n$ belongs to $\mathcal{O}(P_{n} USU^{*}) \in P_{n}$, for each $n$, we can find a unitary $V_n$ in $P_{n} \mathcal{M}P_{n}$ so that 
\[||T_n - V_n  (USU^{*} P_n) V_n^{*}||_1 \leq  ||T_n - V_n (USU^{*} P_n) V_n^{*}|| < \frac{\epsilon}{2^n}.\]
 Letting $V = \sum_{n} V_{n}$, we see that $||T - VUSU^{*}V^{*}||_1 < \epsilon$. Thus, $T$ belongs to $\mathcal{O}(\mathcal{S})$ and we are done. 
\end{proof}

 Another problem in this context is that of characterizing the images of operators, for instance projections, under the conditional expectation onto an atomic masa. In the case of $\bh$, there are subtle index type obstructions that pop up\cite{KPNAS1}. The work of Kadison was recently extended from projections to hermitians with finite spectrum by Bownik and Jasper in \cite{Jas3P}, \cite{BowJas1} and \cite{BowJas2}. The complete characterisation that they obtain, while pleasing, is extremely subtle. In the type $II_{\infty}$ factor case, however, the situation is completely transparent. We first show that any reasonable ``diagonal'' can be lifted to a projection. 

\begin{theorem}
 Let $\mathcal{A}$ be an atomic masa in a type $II_{\infty}$ factor $\mathcal{M}$ and let $A \in \mathcal{A}$ be a positive contraction. Then, there is a projection $P$ in $\mathcal{M}$ so that $E(P) = A$. 
\end{theorem}
\begin{proof}
Write $A = \sum_{\alpha} AQ_{\alpha}$ where $Q_{\alpha}$ are a family of orthgonal finite projections in $\mathcal{A}$. Then, $Q_{\alpha} \mathcal{M}Q_{\alpha}$ is a type $II_1$ factor and we may find a projection $P_{\alpha}$ in $Q_{\alpha} \mathcal{M} Q_{\alpha}$ so that $E(P_{\alpha}) = AQ_{\alpha}$. Then, letting $P = \sum P_{\alpha}$, we have that $E(P) = A$. 
\end{proof}

We now turn things around and ask for a characterization of all possible diagonals of a given projection as well as that of positive operators in general. We use the convention that if a positive operator is not trace class, then it's trace is $\infty$. Argerami and Massey in a recent paper\cite{ArgMas2I} proved approximate theorems in this context, which I am able to improve. First, the result for projections. 
 
\begin{theorem}
Let $P$ be a projection in a type $II_{\infty}$ factor $\mathcal{M}$ and let $A \in \mathcal{A}$ be a positive contraction where $\mathcal{A}$ is an atomic masa. Then, there is a unitary $U$ such that $E(UPU^{*}) = A$ iff $\tau(P) = \tau(A)$ and $\tau(I-P) = \tau(I-A)$. 
\end{theorem}
\begin{proof}
If either $\tau(P)$ or $\tau(I-P)$ is finite, the theorem follows from theorem(\ref{SHTIN}). For the other case, pick an orthogonal family of finite projections $\{R_{\alpha}\}$ in $\mathcal{A}$ summing up to the identity. Decompose $P = \sum_{\alpha} P_{\alpha}$ and $I-P = \sum Q_{\alpha}$ so that $P_{\alpha}$ and $Q_{\alpha}$ are finite projections for every $\alpha$, and such that $\tau(P_{\alpha}) = \tau(AR_{\alpha})$ and $\tau(P_{\alpha}) + \tau(Q_{\alpha}) = \tau(R_{\alpha})$. Pick a unitary $U$ that conjugates $P_{\alpha} + Q_{\alpha}$ onto $R_{\alpha}$ for every $\alpha$ and by theorem(\ref{conj2}), pick unitaries $V_{\alpha}$ in $R_{\alpha} \mathcal{M} R_{\alpha}$ so  that $E(V_{\alpha} U (P_{\alpha} + Q_{\alpha}) U^{*} V_{\alpha}) = AR_{\alpha}$ for every $\alpha$. 

Then, if we let $V = \sum_{\alpha} V_{\alpha}$, we have that 
\[E(VUSU^{*}V^{*}) = A\]
\end{proof}

I now extend the above analysis to general positive operators. Let $S \in \mathcal{M}$ and $A \in \mathcal{A}$ be positive operators. For there to exist a $T$ in $\mathcal{O}(\mathcal{S})$ such that $E(T) = A$, it is necessary that $A \prec S$(see (\ref{Maj2I}) for the definition of majorization between general positive operators in type $II_{\infty}$ factors). However, this is not enough. For example, let $A$ be a projection such that both $A$ and $I-A$ have infinite trace. Let $\{P_r\}$ be a sequence of trace $1$ projections indexed by the rationals in $\mathbb{Q} \cap (0,1)$ summing upto $I$ and let $S$ be the operator $S = \sum_{r \in \mathbb{Q} \cap (0,1)} r P_r$. Then, for both $A$ and $S$, the upper and lower spectral scales are the constant functions $1$ and $0$ respectively. It is easy to see that if there is a positive operator $T$ such that $E(T) = A$, then $T$ must equal $A$. However, $A$ is not in $\mathcal{O}(\mathcal{S})$. 

Let $\mathcal{F}(\mathcal{M})$ be the ideal of $\tau$ finite rank operators, $\mathcal{F}(\mathcal{M}) = \{x \in \mathcal{M} : \tau(x^{*}) < \infty\}$ and let $\mathcal{K}(\mathcal{M}) = \overline{\mathcal{F}(\mathcal{M})}^{||\cdot||}$ be the norm closed two sided ideal of $\tau$ compact operators\cite{Fack}. Let $\mathcal{C}(\mathcal{M})$ be the generalized Calkin algebra $\mathcal{M}/\mathcal{K}(\mathcal{M})$ and let $\sigma_e(S)$ and $\sigma_e(A)$ be the essential spectra of $S$ and $A$, namely the spectra when projected down into $\mathcal{C}(\mathcal{M})$. The majorization relation $A \prec S$ will force $\sigma_e(A) \subset  \operatorname{conv}(\sigma_e(S))$. The above example shows that we need additional constraints on the essential point spectra of $A$ and $S$. We have the following theorem, whose proof is not too hard - It involves a standard cut and paste argument and a use of theorem(\ref{conj2}) and we omit it. 

\begin{theorem}
Let $S $ be a positive operator in a type $II_{\infty}$ factor $\mathcal{M}$ and let $A \in \mathcal{A}$ be a positive operator where $\mathcal{A}$ is an atomic masa. Then, there is a $T$ in $\mathcal{O}(\mathcal{S})$ such that $E(T) = A$ iff 
\begin{enumerate}
 \item We have that $A \prec S$. And further,
 \item If $||\sigma_e(A)|| = ||\sigma_e(S)||$ and if $||\sigma_e(A)||$ belongs to the essential point spectrum of $A$, then it belongs to the essential point spectrum of $S$ as well. And,
 \item If $\alpha_e(A) = \alpha_e(S)$ and if $\alpha_e(A)$ belongs to the essential point spectrum of $A$, then it belongs to the essential point spectrum of $S$ as well.
\end{enumerate}
\end{theorem}

\section{Discussion}
It is routine to extend the Schur-Horn theorem to general finite von Neumann algebras. Let $\mathcal{M}$ be a type $II_1$ von Neumann algebra and let $\mathcal{A}$ be a masa in $\mathcal{M}$. Instead of working with a tracial state, we must now work with the center valued trace $\tau$. Majorization is defined analogously to the case of type $II_1$ factors. The Schur-Horn theorem in this case is
\begin{theorem}
Let $\mathcal{A}$ be a masa in a type $II_1$ von Neumann algebra $\mathcal{M}$. If $A \in \mathcal{A}$ and $S \in \mathcal{M}$ are positive operators with $A \prec S$. Then, there is an element $T \in \mathcal{O}(S)$ such that \[E(T) = A.\]
Alternately, we have that
\[E(\mathcal{O}(S)) = \{A \in \mathcal{A} \mid A \prec S\}\]
\end{theorem}
This can be proved exactly as in the factor case by first getting a local version and then using induction. The proof is a standard application of the direct integral decomposition of $\mathcal{M}$ into type $II_1$ factors and an argument analogous to the proof of theorem(\ref{conj2}) and we omit it.

The situation when it comes to type $III$ factors is far simpler than that for semifinite factors. One point to be noted is that the norm  and SOT closures of the unitary orbits of a hermitian operator in this case are different, unlike the type $II_1$ case. For instance, the norm closure of the unitary orbit of a non-trivial projection is the set of all non-trivial projections, while the SOT closure contains in addition, the identity projection $I$ and the zero projection $0$. We will reserve the term $\mathcal{O}(S)$ for the norm closure of the unitary orbit. The proof of the following is again a simple adaptation of the proof of theorem (\ref{conj2}) and I omit it.

\begin{theorem}
 Let $\mathcal{A}$ be a masa in a type $III$ factor $\mathcal{M}$ that admits a normal conditional expectation. Let $A \in \mathcal{A}$ and $S \in \mathcal{M}$ be positive operators. Then the following are equivalent
\begin{enumerate}
 \item There is an operator $T \in \mathcal{O}(S)$ so that $E(T) = S$.
 \item The following spectral conditions are satisfied
\begin{enumerate}
\item $\sigma(A) \subset \operatorname{conv}(\sigma(S))$.
\item If $||S||$ is in the point spectrum of $A$, then it is also in the point spectrum of $S$. Similarly for $\alpha(\sigma(A))$. 
\end{enumerate}
\end{enumerate}
\end{theorem}


\end{document}